\documentclass[a4paper,11pt]{article}
\usepackage{stmaryrd}
\usepackage{amsfonts}
\usepackage{bbm}
\usepackage{amscd}
\usepackage{mathrsfs}
\usepackage{latexsym,amssymb,amsmath,amscd,amscd,amsthm,amsxtra}
\usepackage[dvips]{graphicx}
\usepackage[utf8]{inputenc}
\usepackage[T1]{fontenc}
\usepackage{enumerate}
\usepackage{lmodern}
\usepackage{amssymb}
\usepackage[all]{xy}
\usepackage{ifpdf}
\ifpdf
\usepackage[colorlinks=true,linkcolor=blue,final,backref=page,hyperindex,citecolor=red]{hyperref}
\else
\usepackage[colorlinks,final,backref=page,hyperindex,hypertex]{hyperref}
\fi

\usepackage{nicefrac,mathtools,enumitem}
\usepackage{microtype}
\usepackage{amsfonts}
\usepackage{amssymb}
\usepackage{mathrsfs}
\usepackage{amsmath}
\usepackage{amsthm}
\usepackage{enumerate}
\usepackage{indentfirst}
\usepackage{amsfonts}
\usepackage{amssymb}
\usepackage{mathrsfs}
\usepackage{amsmath}
\usepackage{amsthm}
\usepackage{enumerate}
\usepackage{cite}
\usepackage{mathrsfs}
\usepackage{geometry}
\allowdisplaybreaks
\newtheorem{thm}{Theorem}[section]
\newtheorem{lem}[thm]{Lemma}

\newtheorem{pro}[thm]{Proposition}

\newtheorem{rmk}[thm]{Remark}
\newtheorem{defi}[thm]{Definition}

\numberwithin{equation}{section}

\newcommand{\be }{\begin{equation}}
\newcommand{\ee }{\end{equation}}

\newcommand{\br}[1]{   [ \cdot,    \cdot  ]   }

\newcommand {\emptycomment}[1]{}

\def\<{\langle}
\def\>{\rangle}

\begin{document}
%%%%%%%%%%%%%%%%%%%%%%%%%%%%
\title{\sf Nijenhuis operators and twisted $\mathcal{O}$-operators on Nambu-Poisson algebras}
\date{}
\author{ \bf Apurba Das$^{1}$\footnote{E-mail: apurbadas348@gmail.com (Corresponding author)}~, \;\;
   Fattoum Harrathi$^{2}$\footnote{E-mail: harrathifattoum285@gmail.com } \;and \; Sami Mabrouk$^{3}$\footnote{E-mail: mabrouksami00@yahoo.fr, sami.mabrouk@fsgf.u-gafsa.tn }
\\ \\
$^{1}${\small Department of Mathematics, Indian Institute of Technology, Kharagpur 721302, West Bengal, India}\\$^{2}${\small University of Gafsa, preparatory Institute for Engineering Studies in Gafsa, 2112 Gafsa, Tunisia}\\$^{3}${\small  University of Gafsa, Faculty of Sciences Gafsa, 2112 Gafsa, Tunisia } 
}

\maketitle

\begin{abstract} 
A ternary Nambu-Poisson algebra (which we call a Nambu-Poisson algebra in the paper) is the underlying algebraic structure of Nambu-Poisson manifolds of order $3$ that appeared in the generalized Hamiltonian mechanics. First, we consider the 2nd cohomology group of a Nambu-Poisson algebra with coefficients in a given representation. Next, we discuss suitable linear deformations of a Nambu-Poisson algebra and show that any such trivial deformation yields a Nijenhuis operator on it. To understand the deformed Nambu-Poisson algebra obtained from a Nijenhuis operator, we introduce a new algebraic structure, which we name NS-Nambu-Poisson algebras. Finally, we consider $\mathcal{O}$-operators twisted by $2$-cocycles and find their close relationships with NS-Nambu-Poisson algebras.
\end{abstract} 

\noindent\textbf{Keywords:} $3$-Lie algebra, Nambu-Poisson algebra, Nijenhuis operator, deformation,  Rota-Baxter operator.

\noindent{\textbf{MSC (2020):}} 17B63, 17A40, 17B38, 17B40.

\tableofcontents

%\maketitle

\section{Introduction}

\subsection{Nambu-Poisson algebras} In 1973, Y. Nambu \cite{N} introduced a generalization of classical Hamiltonian mechanics based on $n$-ary brackets replacing the conventional binary Poisson brackets. The evolution of a system in this generalized Hamiltonian mechanics is governed by the flows of vector fields obtained from $(n-1)$ Hamiltonian functions \cite{T}. Nambu himself proposed an example of such an $n$-ary bracket defined on the space $C^\infty (\mathbb{R}^N)$ of smooth functions on $\mathbb{R}^N$ $(N \geq n):$
\begin{equation}\label{nambu-classical}
\{f_1, \dots, f_n\} = \det\left( \frac{\partial f_i}{\partial x_j} \right)_{i,j=1}^n,
\end{equation}
for $f_1, \ldots, f_n \in C^\infty (\mathbb{R}^N)$. This $n$-ary bracket satisfies the properties analogous to those of the classical Poisson bracket. Namely, the bracket (\ref{nambu-classical}) is skew-symmetric, satisfies the Leibniz rule (with respect to the product of $ C^\infty (\mathbb{R}^N)$) and the following fundamental identity:
\begin{align}\label{fundament}
    \{ f_1, \ldots, f_{n-1} , \{ g_1, \ldots, g_n \} \} = \sum_{i=1}^n \{ g_1, \ldots, g_{i-1}, \{ f_1, \ldots, f_{n-1}, g_i \}, g_{i+1}, \ldots, g_n \},
\end{align}
for all $f_1, \ldots, f_{n-1}, g_1, \ldots, g_n \in C^\infty (\mathbb{R}^N)$, that makes the space $\mathbb{R}^N$ into a Nambu-Poisson manifold of order $n$ \cite{T}. The identity (\ref{fundament}) is an $n$-ary generalization of the classical Jacobi identity, and it is independently studied by Filippov \cite{filippov}. An algebraic abstraction of the structures of $C^\infty (\mathbb{R}^N)$ provided above is known as the Nambu-Poisson algebra of order $n$. Explicitly, a Nambu-Poisson algebra of order $n$ is a commutative associative algebra equipped with an $n$-ary skew-symmetric bracket that satisfies the Leibniz rule and the fundamental identity. It turns out that a Nambu-Poisson algebra of order $n=2$ is nothing but a Poisson algebra. In this paper, we shall focus only on Nambu-Poisson algebras of order $n=3$, which we simply call Nambu-Poisson algebras for our convenience.

\subsection{Nijenhuis operators, NS-algebras and twisted $\mathcal{O}$-operators} Nijenhuis operators were originally investigated by Dorfman \cite{Dorfman} through the lens of Lie algebra deformations. More precisely, any trivial linear deformation of a Lie algebra induces a Nijenhuis operator on it. Conversely, any Nijenhuis operator on a Lie algebra generates a trivial linear deformation. Nijenhuis operators also play a significant role in the theory of integrable systems and nonlinear evolution equations \cite{Dorfman}. Subsequently, Nijenhuis operators and their intimations with trivial linear deformations have been extended to the context of associative algebras, Leibniz algebras, $3$-Lie algebras and to many other algebraic settings \cite{Carinena,fig,guo-lei,LiuShengZhouBai,leroux,ospel}. In \cite{uchino,guo-lei,leroux}, it has been observed that Nijenhuis operators are also useful in constructing new algebraic structures. More precisely, Uchino \cite{uchino} considered NS-algebras as the structure behind Nijenhuis operators on associative algebras. Later, NS-Lie algebras \cite{das} and NS-$3$-Lie algebras \cite{ChtiouiHajjajiMabroukMakhlouf,Hou-Sheng-NSLie} are also investigated while studying Nijenhuis operators on Lie algebras and $3$-Lie algebras, respectively. In \cite{BaishyaDas}, the authors have recently studied Nijenhuis operators on Poisson algebras and investigated the deformed structures. In particular, they introduced NS-Poisson algebras as the underlying structure behind Nijenhuis operators.

In recent years, Rota-Baxter operators and their variants have attracted significant attention due to their connections with the splitting of algebras and the discovery of various pre-algebras \cite{li-guo,CMM,das,das-jhrs,HarrathiSendi,uchino}. Twisted Rota-Baxter operators (and more generally, twisted $\mathcal{O}$-operators) were first considered by Uchino \cite{uchino} in the setting of associative algebras. Since then, these operators have been extensively studied \cite{guo-lei,das-jhrs} and generalized in various algebraic contexts, including Lie algebras \cite{das}, $3$-Lie algebras \cite{ChtiouiHajjajiMabroukMakhlouf,ChtiouiHajjajiMabroukMakhlouf2,Hou-Sheng-NSLie} and Poisson algebras \cite{BaishyaDas}. While NS-algebras arise naturally from Nijenhuis operators, they are more intimately associated with twisted $\mathcal{O}$-operators on associative algebras \cite{das-jhrs,uchino}. More precisely, any twisted $\mathcal{O}$-operator induces an NS-algebra structure, and any NS-algebra arises in this way. In the same spirit, it has been observed that NS-Lie algebras, NS-$3$-Lie algebras and NS-Poisson algebras are closely related to twisted $\mathcal{O}$-operators on Lie algebras, $3$-Lie algebras and Poisson algebras, respectively. A Reynolds operator is a particular case of a twisted Rota-Baxter operator and hence related to NS-algebras \cite{uchino,das}.

\subsection{Overview of the main results} This paper aims to extend the theory of Nijenhuis operators, NS-algebras and twisted $\mathcal{O}$-operators to the framework of Nambu-Poisson algebras. We begin by introducing the second cohomology group of a Nambu-Poisson algebra with coefficients in a representation. Next, we consider $(1,2)$-linear deformations of a Nambu-Poisson algebra. By an $(1,2)$-linear deformation of a Nambu-Poisson algebra, we shall mean a deformation allowing degree $1$ linear deformation of the underlying commutative associative product and degree $2$ linear deformation of the Lie bracket. In particular, we show that a trivial $(1,2)$-linear deformation of a Nambu-Poisson algebra yields a Nijenhuis operator on it. Subsequently, we show that a Nijenhuis operator on a Nambu-Poisson algebra deforms the given Nambu-Poisson structure to a new one. To better understand this deformed Nambu-Poisson structure obtained from a Nijenhuis operator, we introduce the notion of NS-Nambu-Poisson algebras (that can be regarded as the ternary version of NS-Poisson algebras) and show that this new algebraic structure splits Nambu-Poisson algebras. Subsequently, we consider twisted $\mathcal{O}$-operators on a Nambu-Poisson algebra with respect to a given representation and show that such operators naturally induce NS-Nambu-Poisson algebras. Conversely, we observe that any NS-Nambu-Poisson algebra arises in this way.

\subsection{Organization} The paper is organized as follows. In Section \ref{section-2}, we recall some necessary background on Nambu-Poisson algebras and their representations. In Section \ref{section-3}, we consider $(1,2)$-linear deformations of a Nambu-Poisson algebra and show that trivial $(1,2)$-linear deformations induce Nijenhuis operators on the Nambu-Poisson algebra. We introduce and study the notion of NS-Nambu-Poisson algebras in Section \ref{section-4}. NS-Nambu-Poisson algebras are a splitting of Nambu-Poisson algebras. Finally, in Section \ref{section-5}, we first define the second cohomology group of a Nambu-Poisson algebra with coefficients in a given representation. Subsequently, we consider twisted $\mathcal{O}$-operators on Nambu-Poisson algebras and find their close relationships with NS-Nambu-Poisson algebras.

All vector spaces, linear and multilinear maps are over a field $k$ of characteristic $0$.

\section{Some background on Nambu-Poisson algebras }\label{section-2} 
In this section, we recall some necessary definitions regarding Nambu-Poisson algebras and their representations \cite{HarrathiSendi,kac,T,MakhloufAmri, Abramov}. To do this, we begin with some basics of $3$-Lie algebras \cite{fig,filippov}.

\begin{defi}
    A {\bf $3$-Lie algebra} is a pair $(\mathcal{L}, \{ ~, ~, ~ \})$ consisting of a vector space $\mathcal{L}$ endowed with a skew-symmetric trilinear operation $\{ ~, ~, ~ \} : \mathcal{L} \times \mathcal{L} \times \mathcal{L} \rightarrow \mathcal{L}$ satisfying the following fundamental identity:
    \begin{align}\label{funda-iden}
        \{ a, b, \{ c,d , e \} \} = \{ \{ a, b, c \} , d, e \} + \{ c, \{ a, b, d \}, e\} + \{ c, d, \{ a, b, e \} \}, \text{ for } a, b, c, d, e \in \mathcal{L}.
    \end{align}
\end{defi}

Let $(\mathcal{L}, \{ ~, ~, ~ \})$ and $(\mathcal{L}', \{ ~, ~, ~ \}')$ be two $3$-Lie algebras. A {\bf homomorphism} of $3$-Lie algebras from $\mathcal{L}$ to $\mathcal{L}'$ is a linear map $f : \mathcal{L} \rightarrow \mathcal{L}'$ satisfying $f (\{ a, b, c \}) = \{ f(a), f(b) , f(c) \}'$, for all $a, b, c \in \mathcal{L}$.

\begin{defi}
    Let $(\mathcal{L}, \{ ~, ~, ~ \})$ be a $3$-Lie algebra. A {\bf representation} of this $3$-Lie algebra is given by a pair $(V; \rho)$ of a vector space $V$ with a skew-symmetric bilinear map $\rho: \mathcal{L} \times \mathcal{L} \rightarrow \mathrm{End}(V)$ such that for all $a, b, c, d \in \mathcal{L}$, the following identities are hold:
    \begin{align}
        \rho (a, b) \rho (c, d) - \rho (c, d) \rho (a, b) =~& \rho (\{ a, b, c \}, d) + \rho (c, \{ a, b, d \}), \label{repn-1}\\
        \rho (\{ a, b, c \}, d) =~& \rho (a, b) \rho (c, d) + \rho (b, c) \rho (a, d) + \rho (c, a) \rho (b, d).
    \end{align}
\end{defi}

Let $(\mathcal{L}, \{ ~, ~, ~ \})$ be a $3$-Lie algebra. Then it follows that the pair $(\mathcal{L}; \rho_{\mathrm{ad}})$ is a representation, where the map $\rho_\mathrm{ad} : \mathcal{L} \times \mathcal{L} \rightarrow \mathrm{End} (\mathcal{L})$ is given by $\rho_\mathrm{ad} (a, b ) (c) = \{ a, b, c \}$, for $a, b, c \in \mathcal{L}$. This is called the {\em adjoint representation}.

The following result is a characterization of representations of a $3$-Lie algebra.

\begin{pro}\label{3-lie-rep}
    Let $(\mathcal{L}, \{ ~, ~, ~ \})$ be a $3$-Lie algebra. Suppose $V$ is a vector space with a skew-symmetric bilinear map $\rho: \mathcal{L} \times \mathcal{L} \rightarrow \mathrm{End}(V)$. Then $(V; \rho)$ is a representation of the given $3$-Lie algebra if and only if the operation
    \begin{align*}
        \{ (a, u), (b, v), (c, w)  \}_\ltimes := ( \{ a, b, c \} ~ \! , ~ \! \rho (a, b) w + \rho (b, c) u + \rho (c, a) v),
    \end{align*}
    for $(a, u), (b, v), (c, w) \in \mathcal{L} \oplus V$, makes $(\mathcal{L} \oplus V, \{ ~, ~, ~ \}_\ltimes)$ into a $3$-Lie algebra.
\end{pro}

\begin{defi}
    A {\bf ternary Nambu-Poisson algebra} (or simply a {\bf Nambu-Poisson algebra}) is a triple $(A, ~ \! \cdot ~ \! , \{ ~, ~, ~ \})$ consisting of a commutative associative algebra $(A, ~ \! \cdot ~ \!)$ and a $3$-Lie algebra $(A, \{ ~, ~ , ~ \})$ both defined on a same vector space satisfying the following Leibniz rule:
    \begin{align}\label{leibniz-rule-added}
        \{ a, b, c \cdot d \} = \{ a, b, c \} \cdot d ~ \! + ~ \! c \cdot \{ a , b, d \}, \text{ for } a, b, c, d \in A.
    \end{align}
\end{defi}

Let $(A, ~ \! \cdot ~ \! , \{ ~, ~, ~ \})$ and $(A', ~ \! \cdot' ~ \! , \{ ~, ~, ~ \}')$ be two Nambu-Poisson algebras. A {\bf homomorphism} from $A$ to $A'$ is a linear map $f: A \rightarrow A'$ that is both a commutative associative algebra homomorphism and a $3$-Lie algebra homomorphism.

%\textcolor{red}{write examples}

\begin{defi}
    Let $(A, ~ \! \cdot ~ \! , \{ ~, ~, ~ \})$ be a Nambu-Poisson algebra. A {\bf representation} of it, consists of a triple $(V; \mu, \rho)$ of a vector space $V$ with a linear map $\mu : A \rightarrow \mathrm{End}(V)$ and a skew-symmetric bilinear map $\rho: A \times A \rightarrow \mathrm{End}(V)$ such that
    \begin{itemize}
        \item $(V ; \mu)$ is a representation of the commutative associative algebra $(A, ~ \! \cdot ~ \!)$, i.e. $\mu(a \cdot b) = \mu (a) \mu(b)$,
        \item $(V; \rho)$ is a representation of the $3$-Lie algebra $(A, \{ ~, ~, ~\})$,
        \item for all $a, b, c \in A$,
        \begin{align}\label{np3-rep}
            \mu (\{ a, b, c \}) = \rho (a, b) \mu(c) - \mu(c) \rho (a, b) \quad \text{ and } \quad \rho (a, b \cdot c) = \mu(b) \rho (a, c) + \mu(c) \rho (a, b).
        \end{align}
    \end{itemize}
\end{defi}

Let $(A, ~ \! \cdot ~ \! , \{ ~, ~, ~ \})$ be any Nambu-Poisson algebra. We define a linear map $\mu_\mathrm{ad} : A \rightarrow \mathrm{End}(A)$ and a skew-symmetric bilinear map $\rho_\mathrm{ad} :  A \times A \rightarrow \mathrm{End}(A)$ by
\begin{align*}
    (\mu_\mathrm{ad} )(a) b = a \cdot b ~~~\text{ and } ~~~ (\rho_\mathrm{ad}) (a, b) c = \{ a, b, c \}, \text{ for } a, b, c \in A.
\end{align*}
Then the triple $(A; \mu_\mathrm{ad}, \rho_\mathrm{ad})$ is a representation, called the {\em adjoint representation} of the Nambu-Poisson algebra.

Similar to Proposition \ref{3-lie-rep}, here we have the following characterization of a representation.

\begin{pro}
    Let $(A, ~ \! \cdot ~ \! , \{ ~, ~, ~ \})$ be a Nambu-Poisson algebra. Suppose $V$ is a vector space endowed with a linear map $\mu : A \rightarrow \mathrm{End}(V)$ and a skew-symmetric bilinear map $\rho :  A \times A  \rightarrow \mathrm{End}(V)$. Then $(V; \mu, \rho)$ is a representation of the given Nambu-Poisson algebra if and only if the operations
    \begin{align*}
        (a, u) \cdot_\ltimes (b, v) :=~& (a \cdot b ~ \! , ~ \! \mu(a) v + \mu (b) u), \\
        \{ (a, u), (b, v), (c, w) \}_\ltimes:=~&  (\{ a, b, c \} ~\! , ~ \! \rho(a, b) w + \rho (b, c) u + \rho (c, a) v),
    \end{align*}
    for $(a, u), (b, v), (c, w) \in A \oplus V$, makes the triple $(A \oplus V, ~\! \cdot_\ltimes ~\!, \{ ~, ~, ~\}_\ltimes)$ into a Nambu-Poisson algebra.
\end{pro}

Note that the fundamental identity (\ref{funda-iden}) given in the definition of a $3$-Lie algebra can be regarded as the ternary version of the classical Jacobi identity. By fixing one coordinate of the $3$-Lie bracket, one obtains a bilinear skew-symmetric bracket that turns out to satisfy the Jacobi identity. More generally, we have the following result.

\begin{pro}\label{fixing-co}
    Let $(A, ~ \! \cdot ~ \! , \{ ~, ~, ~ \})$ be a Nambu-Poisson algebra. For any fixed $x_0 \in A$, we define a bilinear skew-symmetric bracket $\{ ~, ~ \}_{x_0} : A \times A \rightarrow A$ by $\{ a, b \}_{x_0} := \{ x_0, a, b \}$, for all $a, b \in A$. Then $(A, ~\! \cdot ~\! , \{ ~, ~\}_{x_0})$ is a Poisson algebra.

    Additionally, let $(V; \mu, \rho)$ be a representation of the given Nambu-Poisson algebra. We define a linear map $\rho_{x_0} : A \rightarrow \mathrm{End}(V)$ by $\rho_{x_0} (a) (u) = \rho (x_0, a) u$, for $a \in A$ and $u \in V$. Then $(V; \mu, \rho_{x_0})$ is a representation of the Poisson algebra $(A, ~\! \cdot ~\! , \{ ~, ~\}_{x_0})$.
\end{pro}

\begin{proof}
   For any $a, b, c \in A$, we have
   \begin{align*}
       \{ a, \{ b, c \}_{x_0} \}_{x_0} = \{ x_0, a, \{ x_0, b, c \} \} =~& \{ \{ x_0, a, x_0 \} , b, c \} + \{ x_0 , \{ x_0 , a, b \}, c \} + \{ x_0, b, \{ x_0, a, c \} \} \\
       =~& \{ \{ a, b \}_{x_0} , c \}_{x_0} + \{ b, \{ a, c \}_{x_0} \}_{x_0}
   \end{align*}
   and 
   \begin{align*}
       \{ a, b \cdot c \}_{x_0} = \{ x_0, a, b \cdot c \} = \{ x_0, a , b \} \cdot c ~ \! + ~ \! b \cdot \{ x_0, a, c \} =
       \{ a, b \}_{x_0} \cdot c + b \cdot \{ a, c \}_{x_0}.
   \end{align*}
   This shows that $(A, ~ \! \cdot ~ \! , \{ ~, ~ \}_{x_0})$ is a Poisson algebra. For the second part, we first observe that
   \begin{align*}
       \rho_{x_0} (\{ a, b \}_{x_0} ) = \rho (x_0, \{ x_0, a, b \}) \stackrel{(\ref{repn-1})}{=}~& \rho (x_0, a) \rho (x_0, b) - \rho (x_0, b) \rho (x_0, a) - \rho ( \{ x_0, a, x_0 \}, b) \\
      =~&  \rho_{x_0} (a) \rho_{x_0} (b) - \rho_{x_0} (b) \rho_{x_0} (a)
   \end{align*}
   which shows that $(V ; \rho_{x_0})$ is a representation of the Lie algebra $(A, \{ ~, ~\}_{x_0})$. Additionally, 
   \begin{align*}
       \mu (\{ a, b \}_{x_0}) = \mu  (\{ x_0, a, b \}) = \rho (x_0, a) \mu(b) - \mu(b) \rho (x_0, a) = \rho_{x_0} (a) \mu(b) - \mu (b) \rho_{x_0} (a),
   \end{align*}
   \begin{align*}
        \rho_{x_0} (a \cdot b) = \rho(x_0, a \cdot b) = \mu (a) \rho (x_0, b) + \mu (b) \rho (x_0, a) = \mu(a) \rho_{x_0} (b) + \mu(b) \rho_{x_0} (a),
   \end{align*}
   for any $a, b \in A$. Hence $(V; \mu, \rho_{x_0})$ is a representation of the Poisson algebra $(A, ~\! \cdot ~\! , \{ ~, ~\}_{x_0})$.
\end{proof}

\section{Second cohomology group, linear deformations and Nijenhuis operators}\label{section-3}

In this section, we start by defining the second cohomology group of a Nambu-Poisson algebra with coefficients in a representation. Next, we consider the notion of $(1,2)$-linear deformations of a Nambu-Poisson algebra. We observe that a trivial $(1,2)$-linear deformation of a Nambu-Poisson algebra yields a Nijenhuis operator on it.

\subsection{The second cohomology group of a Nambu-Poisson algebra}

Let $(A, ~ \! \cdot ~ \!)$ be a commutative associative algebra and $(V; \mu)$ be a representation of it. A {\bf Harrison $2$-cocycle} of $A$ with coefficients in the representation $V$ is simply a bilinear map $\varphi : A \times A \rightarrow V$ satisfying
\begin{align*}
    \varphi (a, b)= \varphi (b, a) \quad \text{ and } \quad  \mu (a) \varphi (b, c) - \varphi (a \cdot b , c) + \varphi (a, b \cdot c) - \mu (c) \varphi (a, b) = 0, \text{ for } a, b, c \in A.
\end{align*}
A Harrison $2$-cocycle $\varphi$ is said to be a {Harrison $2$-coboundary} if there exits a linear map $f : A \rightarrow V$ such that 
\begin{align*}
    \varphi (a, b) = \mu(a) f (b) - f(a \cdot b) + \mu(b) f (a), \text{ for all } a, b \in A.
\end{align*}

\begin{pro}\label{harr-2-co}
    Let $(A, ~ \! \cdot ~ \! )$ be a commutative associative algebra and $(V; \mu)$ be a representation of it. Then a bilinear map $\varphi : A \times A \rightarrow V$ is a Harrison $2$-cocycle if and only if the operation
    \begin{align}\label{eqn-semid-2co}
        (a, u) \cdot_{\ltimes_\varphi} (b, v) := (a \cdot b ~\! , ~ \! \mu(a) v + \mu(b) u + \varphi (a, b)), \text{ for } (a, u), (b, v) \in A \oplus V
    \end{align}
    makes $A \oplus V$ into a commutative associative algebra.
\end{pro}

Let $(\mathcal{L}, \{ ~, ~, ~\})$ be a $3$-Lie algebra and $(V; \rho)$ be a representation of it. Then a skew-symmetric trilinear map $\psi : \mathcal{L} \times \mathcal{L} \times \mathcal{L} \rightarrow V$ is said to be a {\bf $2$-cocycle} of the $3$-Lie algebra $\mathcal{L}$ with coefficients in the representation $V$ if it satisfies 
\begin{align}
    \psi (a , b, \{ c, d, e \}) + \rho (a, b) \psi (c, d, e) =~& \psi ( \{a, b, c \} , d, e) + \psi (c, \{ a, b, d \}, e) + \psi (c, d, \{ a, b, e \}) \nonumber \\
    & \quad + \rho (d, e) \psi (a, b, c) + \rho (e, c) \psi (a, b, d) + \rho (c, d) \psi (a, b, e), \label{2co-3}
\end{align}
for all $a, b, c, d,e \in \mathcal{L}$. Further, a $2$-cocycle $\psi$ is said to be a {\em $2$-coboundary} if there exists a linear map $f : \mathcal{L} \rightarrow V$ such that
\begin{align*}
    \psi (a, b, c) = \rho (f(a) , b) c + \rho (a, f (b)) c  +   \rho (a, b) f (c) - f (\{ a, b, c \}), \text{ for all } a, b, c \in \mathcal{L}.
\end{align*}

\begin{pro}\label{3lie-2-co}
    Let $(\mathcal{L}, \{ ~, ~, ~\})$ be a $3$-Lie algebra and $(V; \rho)$ be a representation of it. Then a skew-symmetric trilinear map $\psi : \mathcal{L} \times \mathcal{L} \times \mathcal{L} \rightarrow V$ is a $2$-cocycle of $\mathcal{L}$ with coefficients in the representation $V$ if and only if the ternary bracket
    \begin{align}\label{eqn-semid-2co-lie}
        \{ (a, u), (b, v) , (c, w) \}_{\ltimes_\psi} :=~&  (\{ a, b, c \} ~\! , ~ \! \rho(a, b) w + \rho (b, c) u + \rho (c, a) v + \psi (a, b, c) ),
    \end{align}
    for $(a, u), (b, v), (c, w) \in \mathcal{L} \oplus V$, makes $(\mathcal{L} \oplus V, \{ ~, ~, ~ \}_{\ltimes_\psi})$ into a $3$-Lie algebra.
\end{pro}

Inspired by the above definitions and results, we will now define the second cohomology group of a Nambu-Poisson algebra with coefficients in a representation. To obtain the higher degree cohomology groups, we will consider a bicomplex by mixing the Harrison cochain complex and the cochain complex of the underlying $3$-Lie algebra in a separate article. 

\begin{defi}\label{defi-np2co}
    Let $(A, ~ \! \cdot ~ \!, \{ ~, ~, ~ \})$ be a Nambu-Poisson algebra and $(V; \mu, \rho)$ be a representation of it. A {\bf Nambu-Poisson $2$-cocycle} of $A$ with coefficients in the representation $V$ is a pair $(\varphi, \psi)$, where
    \begin{itemize}
        \item $\varphi : A \times A \rightarrow V$ is a Harrison $2$-cocycle of the commutative associative algebra $(A, ~ \! \cdot ~ \!)$ with coefficients in the representation $(V; \mu)$,
        \item $\psi : A \times A \times A \rightarrow V$ is a $2$-cocycle of the $3$-Lie algebra $(A, \{ ~, ~, ~\})$ with coefficients in the representation $(V; \rho)$,
    \end{itemize}
    such that for all $a, b, c, d \in A$, the following compatibility holds:
    \begin{align}\label{2co-np}
        \psi (a, b, c \cdot d) + \rho (a, b) \varphi (c, d) = \varphi (c, \{ a, b, d \}) + \mu (c) \psi (a, b, d) + \varphi (\{ a, b, c \} , d) + \mu (d) \psi (a, b, c).
    \end{align}
    
\end{defi}

The set of all Nambu-Poisson $2$-cocycles of $A$ with coefficients in the representation $V$ forms an abelian group with respect to the componentwise addition. We denote this abelian group simply by $\mathcal{Z}^2_\mathrm{NP} (A, V)$.
By generalizing Propositions \ref{harr-2-co} and \ref{3lie-2-co} to the Nambu-Poisson context, we obtain the following result.

\begin{pro}
    Let $(A, ~ \! \cdot ~ \!, \{ ~, ~, ~ \})$ be a Nambu-Poisson algebra and $(V; \mu, \rho)$ be a representation of it. Then a pair $(\varphi, \psi)$ of a bilinear map $\varphi : A \times A \rightarrow V$ and a skew-symmetric trilinear map $\psi :  A \times A \times A \rightarrow V$ is a Nambu-Poisson $2$-cocycle of $A$ with coefficients in the representation $V$ if and only if the triple $(A \oplus V, ~ \! \cdot_{\ltimes_\varphi} ~ \! , \{ ~, ~ , ~ \}_{\ltimes_\psi})$ is a Nambu-Poisson algebra, where the operations $\cdot_{\ltimes_\varphi}$ and $\{ ~, ~, ~ \}_{\ltimes_\psi}$ are given in (\ref{eqn-semid-2co}) and (\ref{eqn-semid-2co-lie}), respectively. In this case, the Nambu-Poisson algebra $(A \oplus V, ~ \! \cdot_{\ltimes_\varphi} ~ \! , \{ ~, ~ , ~ \}_{\ltimes_\psi})$ is said to be the $(\varphi, \psi)$-twisted semidirect product and it is denoted by $A \ltimes_{(\varphi, \psi)} V.$
\end{pro}

\begin{lem}\label{2co-lemma}
Let $(A, ~ \! \cdot ~ \!, \{ ~, ~, ~ \})$ be a Nambu-Poisson algebra and $(V; \mu, \rho)$ be a representation of it.
    For any linear map $f : A \rightarrow V$, we define a bilinear map $\varphi_f : A \times A \rightarrow V$ and a trilinear map $\psi_f : A \times A \times A \rightarrow V$ by 
    \begin{align*}
        \varphi_f (a, b) :=~& \mu(a) f (b) - f(a \cdot b) + \mu(b) f (a),\\
        \psi_f (a, b, c) := ~& \rho (a, b) f (c) +  \rho ( b,c) f(a)+ \rho (c ,a) f(b)  - f (\{ a, b, c \}),
    \end{align*}
    for $a, b, c \in A$. Then $(\varphi_f, \psi_f)$ is a Nambu-Poisson $2$-cocycle of $A$ with coefficients in $V$.
\end{lem}

\begin{proof}
    A direct calculation shows that $\varphi_f$ is a Harrison $2$-cocycle of the commutative associative algebra $(A, ~ \! \cdot ~ \!)$ with coefficients in the representation $(V; \mu)$. Similarly, $\psi_f$ is a $2$-cocycle of the $3$-Lie algebra $(A, \{ ~,~, ~ \})$ with coefficients in the representation $(V; \rho)$. Both of these results also follow from the explicit descriptions of $2$-coboundaries in the respective cochain complexes. Thus, it remains to show the compatibility condition (\ref{2co-np}) for the pair $(\varphi_f, \psi_f)$. For any $a, b, c, d \in A$, we observe that
    \begin{align*}
        &\psi_f (a, b, c \cdot d) + \rho (a, b) \varphi_f (c, d ) - \varphi_f  (c, \{ a, b, d \}) - \mu(c) \psi_f (a, b, d) - \varphi_f  (\{ a, b, c \} , d) - \mu (d) \psi_f (a, b, c) \\
        &= \rho (a, b) f (c \cdot d) + \rho (b, c \cdot d) f(a) + \rho (c \cdot d, a) f(b) - f ( \{ a, b, c \cdot d \}) 
        + \rho (a, b) \mu (c) f(d) \\ 
        & \quad - \rho (a, b) f (c \cdot d)  + \rho (a, b) \mu (d) f (c) - \mu (c) f (\{ a, b, d \}) + f (c \cdot \{ a, b, d \}) - \mu (\{ a, b, d \}) f (c) \\
        & \quad - \mu (c) \rho (a, b) f(d) - \mu (c) \rho (b, d) f (a) - \mu (c) \rho (d, a) f (b) + \mu (c) f (\{ a, b, d \}) \\
        & \quad - \mu (\{ a, b, c \}) f (d) + f (\{ a, b, c \} \cdot d)- \mu(d) f(\{ a, b, c \})\\ 
        & \quad -\mu (d) \rho (a, b) f(c) -\mu (d) \rho (b,c) f(a) -\mu (d) \rho (c,a) f(b) + \mu (d)f(\{a,b,c\})\\
        &= 0 \quad (\text{by using the identities in }(\ref{np3-rep}))
    \end{align*}
    which verifies (\ref{2co-np}). Hence, the result is proved.
\end{proof}

A Nambu-Poisson $2$-cocycle is said to be a Nambu-Poisson $2$-coboundary if it is of the form $(\varphi_f, \psi_f)$, for some linear map $f : A \rightarrow V$. We denote the space of all Nambu-Poisson $2$-coboundaries of $A$ with coefficients in $V$ by the notation $\mathcal{B}^2_\mathrm{NP} (A, V)$. Then it follows from Lemma \ref{2co-lemma} that $\mathcal{B}^2_\mathrm{NP} (A, V) \subset \mathcal{Z}^2_\mathrm{NP} (A, V)$. The quotient group
\begin{align*}
    \mathcal{H}^2_\mathrm{NP} (A, V) := \frac{\mathcal{Z}^2_\mathrm{NP} (A, V) }{\mathcal{B}^2_\mathrm{NP} (A, V) }
\end{align*}
is said to be the {\bf second cohomology group} of the Nambu-Poisson algebra $(A, ~ \! \cdot ~ \!, \{ ~, ~ , ~\})$ with coefficients in the representation $(V; \mu, \rho)$.

In Proposition \ref{fixing-co}, we have seen that fixing a coordinate of the $3$-Lie bracket of a Nambu-Poisson algebra gives rise to a Poisson algebra structure. In the following, a similar result has been proved for the corresponding $2$-cocycles (see \cite{BaishyaDas} for the explicit description of a Poisson $2$-cocycle).

\begin{pro}\label{2co-prop}
    Let $(A, ~ \! \cdot ~ \!, \{ ~, ~, ~ \})$ be a Nambu-Poisson algebra and $(V; \mu, \rho)$ be a representation of it. Suppose $(\varphi, \psi)$ is a Nambu-Poisson $2$-cocycle of $A$ with coefficients in $V$. For any fixed $x_0 \in A$, we define a skew-symmetric bilinear map $\psi_{x_0} :  A \times A \rightarrow V$ by $\psi_{x_0} (a, b) = \psi (x_0, a, b)$, for $a, b \in A$. Then the pair $(\varphi, \psi_{x_0})$ is a Poisson $2$-cocycle of the Poisson algebra $(A, ~ \! \cdot ~ \! , \{ ~, ~ \}_{x_0})$ with coefficients in the representation $(V; \mu, \rho_{x_0}).$
\end{pro}

\begin{proof}
For any $a, b, c \in A$, we observe that
\begin{align*}
    &(\rho_{x_0} (a)) \psi_{x_0} (b, c) + (\rho_{x_0} (b)) \psi_{x_0} (c, a) + (\rho_{x_0} (c)) \psi_{x_0} (a, b) \\
    & \qquad + \psi_{x_0} (a, \{ b, c \}_{x_0} ) + \psi_{x_0} (b, \{ c, a \}_{x_0}) + \psi_{x_0} (c, \{ a, b \}_{x_0} ) \\
    &= \rho (x_0, a) \psi (x_0, b, c) + \rho (x_0, b) \psi (x_0, c, a) + \rho (x_0, c) \psi (x_0, a, b)  \\
    & \qquad + \psi (x_0, a, \{ x_0, b, c \}) + \psi (x_0, b, \{ x_0, c, a \}) + \psi (x_0, c, \{ x_0, a, b \} ) = 0 \quad (\text{using } (\ref{2co-3})).
\end{align*}
Hence $\psi_{x_0}$ is a $2$-cocycle of the Lie algebra $(A, \{ ~, ~ \}_{x_0})$ with coefficients in the representation $(V; \rho_{x_0})$. Moreover, we also have
\begin{align*}
    &\psi_{x_0} (a, b \cdot c) + (\rho_{x_0} (a)) \varphi (b, c) \\
    &= \psi (x_0, a, b \cdot c) + \rho (x_0, a) \varphi (b, c) \\
    &= \varphi (b, \{ x_0, a, c \}) + \mu(b) \psi (x_0, a, c ) + \varphi (\{ x_0, a, b \}, c) + \mu (c) \psi (x_0, a, b) \quad (\text{by } (\ref{2co-np}))\\
    &= \varphi (b, \{ a, c \}_{x_0}) + \mu (b) \psi_{x_0} (a, c) + \varphi (\{ a, b \}_{x_0}, c) + \mu (c) \psi_{x_0} (a, b).
\end{align*}
This proves that $(\varphi, \psi_{x_0})$ is a Poisson $2$-cocycle.
\end{proof}

The above result shows that Definition \ref{defi-np2co} of a Nambu-Poisson $2$-cocycle is compatible with Poisson $2$-cocycles.

\subsection{Linear deformations and Nijenhuis operators}

\begin{defi}
     Let $(A, ~ \! \cdot ~ \!, \{ ~, ~, ~ \})$ be a Nambu-Poisson algebra. An {\bf $(1,2)$-linear deformation} of $A$ consists of two parametrized operations $~\cdot_t : A \times A \rightarrow A$ and $\{ ~, ~, ~ \}_t : A \times A \times A \rightarrow A$ of the form
     \begin{align}\label{parametrized}
         a \cdot_t  b = a \cdot b + t ~ \! \varphi_1 (a, b) ~~~~ \text{ and } ~~~~ \{ a, b, c \}_t = \{ a, b,c \} + t ~ \!  \psi_1 (a, b, c) + t^2 ~\!  \psi_2 (a, b, c )
     \end{align}
     (where $\varphi_1 : A \times A \rightarrow A$ is a symmetric bilinear map and $\psi_1, \psi_2 : A \times A \times A \rightarrow A$ are skew-symmetric trilinear maps) that makes the triple $(A, ~\! \cdot_t ~ \! , \{ ~, ~, ~\}_t)$ into a Nambu-Poisson algebra.
\end{defi}

In this case, we often say that $(\varphi_1 ; \psi_1, \psi_2)$ generates the $(1,2)$-linear deformation $(A, ~\! \cdot_t ~ \! , \{ ~, ~, ~\}_t)$.

\begin{thm}
     Let $(A, ~ \! \cdot ~ \!, \{ ~, ~, ~ \})$ be a Nambu-Poisson algebra. Suppose there is a symmetric bilinear map $\varphi_1: A \times A \rightarrow A$ and skew-symmetric trilinear maps $\psi_1, \psi_2: A \times A \times A \rightarrow A$. Then $(\varphi_1; \psi_1, \psi_2)$ generates an $(1, 2)$-linear deformation of the given Nambu-Poisson algebra if and only if the following conditions hold:
     \begin{itemize}
     \item[(i)] $(\varphi_1, \psi_1)$ is a Nambu-Poisson $2$-cocycle of the given Nambu-Poisson algebra $(A, ~ \! \cdot ~ \!, \{ ~, ~, ~ \})$ with coefficients in the adjoint representation,
    \item[(ii)] $(A, \varphi_1, \psi_2)$ is a Nambu-Poisson algebra,
    \item[(iii)] for all $a, b, c, d, e \in A$,
     \end{itemize}
     \begin{align}\label{thm-first-eq}
         &\psi_2 (a, b, \{ c, d, e \}) + \{ a, b, \psi_2 (c, d, e ) \} + \psi_1 (a, b, \psi_1 (c, d, e )) 
         = \psi_2 (\{ a, b, c \}, d, e ) \nonumber \\
         & \qquad + \{ \psi_2 (a, b, c), d, e \} + \psi_1 (\psi_1 (a, b, c), d, e )  
         + \psi_2 (c, \{ a, b, d \}, e) + \{ c, \psi_2 (a, b, d ), e \} \nonumber \\ & \qquad + \psi_1 (c, \psi_1 (a, b, d), e ) 
         +\psi_2 (c, d, \{ a, b, e \}) + \{ c, d , \psi_2 (a, b, e ) \} + \psi_1 (c, d, \psi_1 (a, b, e )),
     \end{align}
     \begin{align}\label{thm-second-eq}
         &\psi_1 (a, b, \psi_2 (c, d, e )) + \psi_2 (a, b, \psi_1 (c, d, e )) = \psi_1 (\psi_2 (a, b, c) , d, e) + \psi_2 (\psi_1 (a, b, c), d, e)  \nonumber \\
        & \quad + \psi_1 (c, \psi_2 (a, b, d) , e ) + \psi_2 (c, \psi_1 (a, b, d) , e ) + \psi_1 (c, d , \psi_2 (a, b, e) ) + \psi_2 (c, d, \psi_1 (a, b, e )),
     \end{align}
     \begin{align}\label{thm-third-eq}
         \psi_2 (a, b, c \cdot d) + \psi_1 (a, b,  \varphi_1 (c, d)) =~& \varphi_1 (c, \psi_1 (a, b, d) ) + c \cdot \psi_2 (a, b, d )  \nonumber \\
         &+ \varphi_1 ( \psi_1 (a, b, c) , d) + \psi_2 (a, b, c) \cdot d.
     \end{align}
\end{thm}

\begin{proof}
    Consider the parametrized operations $\cdot_t : A \times A \rightarrow A$ and $\{ ~, ~, ~\}_t : A \times A \times A \rightarrow A$ given in (\ref{parametrized}). Then $(A, ~ \! \cdot_t ~ \!)$ is a commutative associative algebra if and only if $(a \cdot_t  b) \cdot_t  c = a \cdot_t (b \cdot_t c)$, for all $a, b, c \in A$. This is equivalent to the conditions (which are obtained by comparing the corresponding coefficients of $t^i$, for $i=1,2$)
    \begin{align}
        \varphi_1 (a \cdot b ,c) + \varphi_1 (a, b) \cdot c =~& \varphi_1 (a, b \cdot c) + a \cdot \varphi_1 (b, c), \label{ass-co1}\\
        \varphi_1 (\varphi_1 (a, b), c) =~& \varphi_1 (a, \varphi_1 (b, c)),\label{ass-co2}
    \end{align}
    for all $a, b, c \in A.$ Next, $(A, \{ ~, ~, ~ \}_t)$ is a $3$-Lie algebra if and only if the fundamental identity holds
    \begin{align*}
        \{ a, b, \{ c, d, e \}_t \}_t = \{ \{ a, b, c \}_t, d , e \}_t + \{ c, \{ a, b, d \}_t, e \}_t + \{ c, d, \{ a, b, e \}_t \}_t,
    \end{align*}
    for all $a, b, c, d, e \in A$. In the same way, this is equivalent to the conditions:
    \begin{align}\label{eqn-c}
        &\psi_1 (a, b, \{ c, d, e \}) + \{ a, b, \psi_1 (c, d, e ) \}= \psi_1 (\{ a, b, c \}, d, e ) + \{ \psi_1 (a, b, c) , d, e \} \nonumber \\
        & \qquad \quad + \psi_1 (c, \{ a, b, d \}, e) + \{ c, \psi_1 (a, b, d), e \} + \psi_1 (c, d, \{ a, b, e \}) + \{ c, d , \psi_1 (a, b, e ) \},
    \end{align}
    \begin{align}\label{eqn-f}
        \psi_2 (a, b, \psi_2 (c, d, e) ) = \psi_2 (\psi_2 (a, b, c), d, e ) + \psi_2 (c, \psi_2 (a, b, d), e) + \psi_2 (c, d, \psi_2 (a, b, e))
    \end{align}
    together with the identities (\ref{thm-first-eq}) and (\ref{thm-second-eq}). Additionally, the Leibniz rule 
    \begin{align*}
    \{ a, b, c \cdot_t d \}_t = \{ a, b, c \}_t \cdot_t d + c \cdot_t \{ a, b, d \}_t
    \end{align*}
    holds if and only if 
    \begin{align}\label{eqn-g}
        \psi_1 (a, b, c \cdot d) + \{ a, b, \varphi_1 (c, d ) \} = \varphi_1 (\{ a, b, c \}, d) + \psi_1 (a, b, c ) \cdot d + \varphi_1 (c , \{ a, b, d \}) + c \cdot \psi_1 (a, b, d ),
    \end{align}
    \begin{align}\label{eqn-i}
        \psi_2 (a, b, \varphi_1 (c, d) ) = \varphi_1 (\psi_2 (a, b, c), d) + \varphi_1 (c, \psi_2 (a, b, d )),
    \end{align}
    and the identity (\ref{thm-third-eq}) holds. We observe that the conditions (\ref{ass-co1}), (\ref{eqn-c}) and (\ref{eqn-g}) imply that the pair $(\varphi_1, \psi_1)$ is a Nambu-Poisson $2$-cocycle of the Nambu-Poisson algebra $(A, ~ \! \cdot ~ \!, \{ ~, ~, ~ \})$ with coefficients in the adjoint representation. On the other hand, the conditions (\ref{ass-co2}), (\ref{eqn-f}) and (\ref{eqn-i}) imply that the triple $(A, \varphi_1, \psi_2)$ is a Nambu-Poisson algebra. This completes the proof.
\end{proof}

\begin{defi}
     Let $(A, ~ \! \cdot ~ \!, \{ ~, ~, ~ \})$ be a Nambu-Poisson algebra. Two $(1,2)$-linear deformations $(A, ~\! \cdot_t ~\! , \{ ~, ~, ~ \}_t)$ and $(A, ~\! \cdot'_t ~\! , \{ ~, ~, ~ \}'_t)$ are said to be {\bf equivalent} if there exists a linear map $N : A \rightarrow A$ such that the parametrized map $(\mathrm{Id} + t N) : A \rightarrow A$ is a homomorphism of Nambu-Poisson algebras from $(A, ~\! \cdot_t ~\! , \{ ~, ~, ~ \}_t)$ to $(A, ~\! \cdot'_t ~\! , \{ ~, ~, ~ \}'_t)$.
\end{defi}

Note that the map $(\mathrm{Id} + tN) : A \rightarrow A$ is a homomorphism of commutative associative algebras from $(A, ~ \! \cdot_t ~ \!)$ to $(A, ~ \! \cdot'_t ~ \!)$, is equivalent to
\begin{align}
    \varphi_1 (a, b) - \varphi_1' (a, b) =~& N(a) \cdot b + a \cdot N (b) - N (a \cdot b), \label{equiv-1} \\
    N (\varphi_1 (a, b)) =~& \varphi_1' (N (a), b) + \varphi_1' (a, N (b) ) + N(a) \cdot N (b), \label{equiv-2}\\
    \varphi_1' (N(a), N(b) ) =~& 0.
\end{align}
Similarly, $(\mathrm{Id}+ t N)$ is also a homomorphism of $3$-Lie algebras from $(A, \{ ~, ~, ~ \}_t)$ to $(A, \{ ~, ~, ~ \}'_t)$, is equivalent to the conditions
\begin{align}
    &\psi_1 (a, b, c) - \psi_1' (a, b, c) = \{ N (a), b, c \} + \{ a, N (b), c \} + \{ a, b, N (c) \} - N \{ a, b, c\}, \label{equiv-4}\\
   & \psi_2 (a, b, c) - \psi_2' (a, b, c) = \{ N (a), N (b), c \} + \{ N (a), b, N (c) \} + \{ a, N (b), N (c) \} \nonumber \\
    & \qquad \qquad \qquad \qquad \qquad + \psi_1' (N (a), b, c) + \psi_1' (a, N (b), c ) + \psi_1' (a, b, N (c)) - N (\psi_1 (a, b, c)), \label{equiv-5}\\
   & \qquad  N (\psi_2 (a, b, c)) = \psi_1' (N (a) , N (b), c ) + \psi_1' (N (a), b, N (c)) + \psi_1' (a, N (b), N(c)) \nonumber \\
    &\qquad \qquad \qquad \qquad+ \psi_1' (N (a), b, c) + \psi_1' (a, N (b), c) + \psi_1' (a, b, N (c)) + \{ N(a), N(b), N(c) \}, \label{equiv-6}\\
    &\psi_1' (N(a), N(b) , N(c)) + \psi_2' (N(a), N(b), c) + \psi_2' (N(a), b, N(c)) + \psi_2' (a, N(b), N(c)) = 0, \label{equiv-7}\\
    & \qquad \qquad \qquad \psi_2' (N(a), N(b), N(c) ) = 0, \label{equiv-8}
\end{align}
for all $a, b, c \in A$.

\begin{defi}
    Let $(A, ~ \! \cdot ~ \!, \{ ~, ~, ~ \})$ be a Nambu-Poisson algebra. An $(1, 2)$-linear deformation $(A, ~\! \cdot_t ~\! , \{ ~, ~, ~ \}_t)$ is said to be {\bf trivial} if it is equivalent to the undeformed one.
\end{defi}

Hence for a trivial $(1,2)$-linear deformation $(A, ~\! \cdot_t ~\! , \{ ~, ~, ~ \}_t)$, all the identities (\ref{equiv-1})-(\ref{equiv-8}) hold when we take $\varphi_1' = 0$ and $\psi_1' = \psi_2' = 0$. In these identities, $N$ is a linear map such that $(\mathrm{Id} + t N) : A \rightarrow A$ provides an equivalence of the deformation $(A, ~\! \cdot_t ~\! , \{ ~, ~, ~ \}_t)$ with the undeformed one. Hence, it follows from (\ref{equiv-1}) and (\ref{equiv-2}) (by taking $\varphi_1' = 0$) that
\begin{align}
    \varphi_1 (a, b) =~& N(a) \cdot b + a \cdot N (b) - N (a \cdot b), \label{nij-defor1}\\
    N (\varphi_1 (a, b)) =~& N(a) \cdot N (b).
\end{align}
Similarly, it follows from (\ref{equiv-4}), (\ref{equiv-5}) and (\ref{equiv-6}) (by taking $\psi_1' = \psi_2' = 0$) that
\begin{align}
    \psi_1 (a, b, c) =~&  \{ N (a), b, c \} + \{ a, N (b), c \} + \{ a, b, N (c) \} - N \{ a, b, c\}, \\
    \psi_2 (a, b, c) =~& \{ N (a), N (b), c \} + \{ N (a), b, N (c) \} + \{ a, N (b), N (c) \} - N (\psi_1 (a, b, c)), \\
    N (\psi_2 (a, b, c)) =~& \{ N (a), N(b), N(c) \}, \text{ for } a, b, c \in A. \label{nij-defor5}
\end{align}

%In this case, we often say that $(\mu_1; \varphi_1, \varphi_2)$ generates the $(1,2)$-linear deformation $(A, ~\! \cdot_t ~ \! , \{ ~, ~, ~\}_t)$.

We will now introduce Nijenhuis operators on a Nambu-Poisson algebra and consider some basic properties.

\begin{defi}
    Let $(A, ~ \! \cdot ~ \!, \{ ~, ~, ~ \})$ be a Nambu-Poisson algebra. A linear map $N: A \rightarrow A$ is said to be a {\bf Nijenhuis operator} on $A$ if
    \begin{align}
        N(a) \cdot N (b) =~& N \big( N(a) \cdot b + a \cdot N (b) - N (a \cdot b) \big), \label{nij-com}\\
        \{ N (a), N (b), N (c) \} =~& N \big( \{ N (a), N (b), c \} + \{ N (a) , b, N (c) \} + \{ a, N (b), N (c) \} \nonumber \\
        & \qquad - N ( \{ N (a), b, c \} + \{ a, N (b), a \} + \{ a, b, N (c) \} - N \{ a, b, c \} ) \big), \label{nij-3-lie}
    \end{align}
    for all $a, b, c \in A$. The first condition is equivalent to saying that $N$ is a Nijenhuis operator on the commutative associative algebra $(A, ~ \! \cdot ~ \! )$ while the second condition is equivalent to saying that $N$ is a Nijenhuis operator on the $3$-Lie algebra $(A, \{ ~, ~, ~ \})$.
\end{defi}

\begin{rmk}
    Let $(A, ~ \! \cdot_t ~ \!, \{ ~, ~, ~ \}_t)$ be a trivial $(1,2)$-linear deformation of a Nambu-Poisson algebra $(A, ~ \! \cdot ~ \!, \{ ~, ~, ~ \})$, and let $(\mathrm{Id} + tN) : A \rightarrow A$ defines an equivalence of the deformation $(A, ~ \! \cdot_t ~ \!, \{ ~, ~, ~ \}_t)$ with the undeformed one. Then it follows from (\ref{nij-defor1})-(\ref{nij-defor5}) that the map $N : A \rightarrow A$ is a Nijenhuis operator on the Nambu-Poisson algebra $(A, ~ \! \cdot ~ \!, \{ ~, ~, ~ \}).$
\end{rmk}

It is well-known that the higher powers of a Nijenhuis operator (defined on a Lie algebra) are also Nijenhuis operators. This result is also known for Nijenhuis operators defined on commutative associative algebras \cite{Carinena} and on $3$-Lie algebras \cite{LiuShengZhouBai}. Combining their results, we get the following.

\begin{pro}\label{prop-higher-nij}
    Let $(A, ~ \! \cdot ~ \!, \{ ~, ~, ~ \})$ be a Nambu-Poisson algebra and $N : A \rightarrow A$ be a Nijenhuis operator on it. Then for any $k \geq 0$, the map $N^k: A \rightarrow A$ is also a Nijenhuis operator on the given Nambu-Poisson algebra.
\end{pro}

\begin{thm}\label{thm-deformed-np}
    Let $(A, ~ \! \cdot ~ \!, \{ ~, ~, ~ \})$ be a Nambu-Poisson algebra and $N : A \rightarrow A$ be a Nijenhuis operator on it. Then the triple $A_N = (A, ~ \! \cdot_N ~ \! , \{ ~, ~, ~ \}_N)$ is a new Nambu-Poisson algebra, where
    \begin{align*}
        a \cdot_N b :=~& N(a) \cdot b + a \cdot N (b) - N (a \cdot b), \\
        \{ a, b, c \}_N :=~&  \{ N (a), N (b), c \} + \{ N (a) , b, N (c) \} + \{ a, N (b), N (c) \} \\
        & \qquad - N ( \{ N (a), b, c \} + \{ a, N (b), a \} + \{ a, b, N (c) \} - N \{ a, b, c \} ),
    \end{align*}
    for $a, b, c \in A$. Moreover, $N: A_N \rightarrow A$ is a homomorphism of Nambu-Poisson algebras.
\end{thm}

\begin{proof}
    Since $N$ is a Nijenhuis operator on the commutative associative algebra $(A, ~ \! \cdot ~ \! )$, it follows that $(A, ~ \! \cdot_N ~ \!)$ is also a commutative associative algebra \cite{Carinena}. Also, $N$ is a Nijenhuis operator on the $3$-Lie algebra $(A, \{ ~, ~, ~ \})$ implies that $(A, \{ ~,~, ~ \}_N)$ is also a $3$-Lie algebra \cite{LiuShengZhouBai}. Next, a direct (but long) calculation shows that the Leibniz rule
    \begin{align*}
        \{ a, b, c \cdot_N d \}_N = \{ a, b, c \}_N \cdot_N d + c \cdot_N \{ a, b, d \}_N, \text{ for } a, b, c , d \in A
    \end{align*}
    holds, which shows that $(A, ~ \! \cdot_N ~ \! , \{ ~, ~, ~\}_N)$ is a Nambu-Poisson algebra. Finally, the last part follows from the definition of the Nijenhuis operator.
\end{proof}

The Nambu-Poisson algebra $(A, ~ \! \cdot_N ~ \! , \{ ~, ~, ~\}_N)$ constructed in the previous theorem is said to be the {\em deformed Nambu-Poisson algebra} (induced by the Nijenhuis operator $N$). To understand some tedious calculations present in the above result, we introduce the notion of an NS-Nambu-Poisson algebra in the next section. 

\section{NS-Nambu-Poisson algebras}\label{section-4}

In this section, we introduce the notion of an NS-Nambu-Poisson algebra and study some basic properties. We observe that an NS-Nambu-Poisson algebra naturally induces a subadjacent Nambu-Poisson algebra. On the other hand, we show that any Nijenhuis operator on a Nambu-Poisson algebra gives rise to an NS-Nambu-Poisson algebra structure.  

\begin{defi}\label{BaishyaDas,leroux}
    An {\bf NS-commutative algebra} is a triple $(A, \diamond , *)$ consisting of a vector space $A$ endowed with two bilinear operations $\diamond, * : A \times A \rightarrow A$ in which $*$ is commutative and the following conditions hold:
    \begin{align*}
        a \diamond (b \diamond c) =~& (a \odot b) \diamond c,\\
        a \diamond (b * c) + a * (b \odot c) =~& b \diamond (a * c) + b * (a \odot c),
    \end{align*}
    for all $a, b, c \in A$. Here, we have used the notation 
    \begin{align}\label{operation-odot}
        a \odot b = a \diamond b + b \diamond a + a * b.
    \end{align}
\end{defi}

In the above definition, if the operation $\diamond$ is trivial, then $(A, *)$ becomes a commutative associative algebra. On the other hand, if $*$ is trivial, then $(A, \diamond)$ turns out to be a zinbiel algebra.

The following result is well-known \cite{uchino,BaishyaDas}.

\begin{pro}\label{prop-total-nsc}
    Let $(A, \diamond, *)$ be an NS-commutative algebra. 
    \begin{itemize}
        \item[(i)] Then $(A, \odot)$ is a commutative associative algebra, called the subadjacent algebra.
        \item[(ii)] We define a linear map $\mu_\diamond : A \rightarrow \mathrm{End}(A)$ by $\mu_\diamond (a) b := a \diamond b$, for $a, b \in A$. Then $(A; \mu_\diamond)$ 
     is a representation of the subadjacent algebra $(A, \odot)$. Moreover, the bilinear map 
     \begin{align*}
        \varphi_* : A \times A \rightarrow A ~~~ \text{ given by } ~~~  \varphi_* (a , b) = a * b, \text{ for } a, b \in A
     \end{align*}
     is a Harrison $2$-cocycle of $(A, \odot)$ with coefficients in the representation $(A; \mu_\diamond)$.
    \end{itemize}
\end{pro}

\begin{defi}
    An {\bf NS-$3$-Lie algebra} is a triple $(A, [~, ~, ~ ], \llbracket ~, ~, ~ \rrbracket)$ of a vector space $A$ with trilinear operations $[~,~,~],\llbracket ~, ~, ~ \rrbracket : A \times A \times A \rightarrow A$ in which $[~,~, ~ ]$ is skew-symmetric on the first two arguments and $\llbracket ~, ~, ~ \rrbracket$ is completely skew-symmetric such that for $a, b, c, d, e \in A$, the following set of identities hold:
    \begin{align*}
        [a, b, [c, d, e ]] =~& [ \{ \! \! \{ a, b, c \} \! \! \}, d, e] + [c, \{ \! \! \{ a, b, d \} \! \! \}, e] + [c, d, [a, b, e]], \\
        [ \{ \! \! \{ a, b, c \} \! \! \}, d, e] =~& [a, b, [c, d, e]] + [b, c, [a, d, e]] + [c, a, [b,d,e]], \\
        \llbracket a, b, \{ \! \! \{ c, d, e \} \! \! \} \rrbracket + [a, b, \llbracket c, d, e \rrbracket ] =~& \llbracket  \{ \! \! \{ a, b, c    \} \! \! \}, d, e \rrbracket + \llbracket c,  \{ \! \! \{ a, b, d \} \! \! \}, e \rrbracket 
        + \llbracket c, d,  \{ \! \! \{ a, b, e \} \! \! \} \rrbracket\\ &+ [ d, e, \llbracket a, b, c \rrbracket ] + [ e, c, \llbracket a, b, d \rrbracket ] + [ c, d, \llbracket a, b, e \rrbracket ].
    \end{align*}
    Here, we have used the notation 
    \begin{align}\label{total-3lie}
         \{ \! \! \{ a, b, c \} \! \! \} = [a, b, c ] + [b, c, a]+ [c, a, b] + \llbracket a, b, c \rrbracket.
    \end{align}
\end{defi}

In the above definition, if $[~,~,~]$ is trivial then $(A, \llbracket ~, ~, ~ \rrbracket)$ becomes a $3$-Lie algebra. On the other hand, if the operation $\llbracket ~, ~, ~ \rrbracket$ is trivial, then $(A, [~,~,~])$ turns out to be a pre-$3$-Lie algebra in the sense of \cite{jha}. Thus, NS-$3$-Lie algebras unify both $3$-Lie algebras and pre-$3$-Lie algebras.

The following result has been proved in \cite{ChtiouiHajjajiMabroukMakhlouf,Hou-Sheng-NSLie}.

\begin{pro}\label{prop-total-ns3}
    Let $(A, [~, ~, ~ ], \llbracket ~, ~, ~ \rrbracket)$ be an NS-$3$-Lie algebra. 
    \begin{itemize}
        \item[(i)] Then $(A, \{ \! \! \{ ~, ~, ~ \} \! \! \})$ is a $3$-Lie algebra, where $\{ \! \! \{ ~, ~, ~\} \! \! \}$ is given in (\ref{total-3lie}). (This is called the subadjacent $3$-Lie algebra)
        \item[(ii)] We define a skew-symmetric bilinear map $\rho_{[~, ~, ~]} : A \times A \rightarrow \mathrm{End}(A) $ and a skew-symmetric trilinear map $\psi_{ \llbracket ~, ~, ~ \rrbracket } : A \times A \times A \rightarrow A$ by
        \begin{align*}
            (\rho_{[~, ~, ~]} (a, b)) c = [a, b, c ] ~~~  \text{ and } ~~~ \psi_{\llbracket ~, ~, ~ \rrbracket } (a, b, c) = \llbracket a, b, c \rrbracket, \text{ for } a, b, c \in A.
        \end{align*}
        Then $(A; \rho_{[~, ~, ~]})$ is a representation of the subadjacent $3$-Lie algebra $(A, \{ \! \! \{ ~, ~, ~ \} \! \! \})$. Additionally, the map $\psi_{\llbracket ~, ~, ~ \rrbracket}$ is a $2$-cocycle of $(A, \{ \! \! \{ ~, ~, ~ \} \! \! \})$ with coefficients in the representation $(A; \rho_{[~, ~, ~]})$.
    \end{itemize}
\end{pro}

We are now in a position to define the notion of an NS-Nambu-Poisson algebra.

\begin{defi}
    An {\bf NS-Nambu-Poisson algebra} is a quintuple $(A, \diamond, *, [~,~,~], \llbracket ~, ~, ~ \rrbracket)$, where
    \begin{itemize}
        \item $(A, \diamond, *)$ is an NS-commutative associative algebra,
        \item $(A, [~,~,~], \llbracket ~, ~, ~ \rrbracket)$ is an NS-$3$-Lie algebra,
    \end{itemize}
    such that for all $a, b, c, d \in A$, the following compatibilities hold:
    \begin{align}
        \{ \! \! \{ a, b, c \} \! \! \} \diamond d =~& [a, b, c \diamond d ] - c \diamond [a, b, d ], \label{ns-np1}\\
        [a \odot b,  c, d] =~& a \diamond [b, c, d ] + b \diamond [a, c, d], \label{ns-np2} \\
        \llbracket a, b, c \odot d \rrbracket + [a, b, c * d] =~& c * \{ \! \! \{ a, b, d \} \! \! \} + c \diamond \llbracket a, b, d \rrbracket + \{ \! \! \{ a, b, c \} \! \! \} * d + d \diamond \llbracket a, b, c \rrbracket, \label{ns-np3}
    \end{align}
    where the operations $\odot$ and $\{ \! \! \{ ~, ~, ~ \} \! \! \}$ are given in (\ref{operation-odot}) and (\ref{total-3lie}), respectively.
\end{defi}

\begin{rmk}
    In an NS-Nambu-Poisson algebra  $(A, \diamond, *, [~,~,~], \llbracket ~, ~, ~ \rrbracket)$, if the operations $\diamond$ and $[~,~, ~]$ are trivial, then $(A, *, \llbracket ~,~,~ \rrbracket)$ simply becomes a Nambu-Poisson algebra. On the other hand, if the operations $*$ and $\llbracket ~,~,~ \rrbracket$ are trivial, then $(A, \diamond, [~,~,~])$ becomes a pre-Nambu-Poisson algebra introduced in \cite{HarrathiSendi} in the study of the Nambu-Poisson Yang-Baxter equation. Thus, NS-Nambu-Poisson algebras unify both Nambu-Poisson algebras and pre-Nambu-Poisson algebras.
\end{rmk}

In the following result, we show that an NS-Nambu-Poisson algebra induces a Nambu-Poisson algebra structure, a suitable representation of it and a Nambu-Poisson $2$-cocycle. 

\begin{thm}\label{ns-np-thm-2}
    Let $(A, \diamond, *, [~,~,~], \llbracket ~, ~, ~ \rrbracket)$ be an NS-Nambu-Poisson algebra.
\begin{itemize}
    \item[(i)] Then $(A, \odot, \{ \! \! \{ ~, ~, ~ \} \! \! \})$ is a Nambu-Poisson algebra, where the operations $\odot$ and $\{ \! \! \{ ~, ~, ~ \} \! \! \}$ are given in (\ref{operation-odot}) and (\ref{total-3lie}), respectively. (This is called the subadjacent Nambu-Poisson algebra).
    \item[(ii)] The triple $(A; \mu_\diamond, \rho_{[~,~,~]})$ is a representation of the subadjacent Nambu-Poisson algebra $(A, \odot, \{ \! \! \{ ~, ~, ~ \} \! \! \})$. Moreover, $(\varphi_*, \psi_{\llbracket ~, ~, ~ \rrbracket})$ is a Nambu-Poisson $2$-cocycle of the subadjacent Nambu-Poisson algebra with coefficients in the above representation $(A; \mu_\diamond, \rho_{[~,~,~]})$.
\end{itemize}
\end{thm}

\begin{proof}
    (i)  Since $(A, \diamond, *)$ is an NS-commutative algebra, it follows from Proposition \ref{prop-total-nsc} (i) that $(A, \odot)$ is a commutative associative algebra. Similarly, $(A, [~,~,~], \llbracket ~, ~, ~ \rrbracket)$ is an NS-$3$-Lie algebra implies that $(A, \{ \! \! \{ ~, ~, ~ \} \! \! \})$ is a $3$-Lie algebra (cf. Proposition \ref{prop-total-ns3} (i)). Thus, it remains to verify the Leibniz rule. For any $a, b, c, d \in A$, we observe that
    \begin{align}\label{leibniz-left-side}
        \{ \! \! \{ a, b, c \odot d \} \! \! \} = [a, b, c \diamond d] + [a, b, d \diamond c] + [a, b, c * d] + [b, c \odot d, a] + [c \odot d, a, b] + \llbracket a, b, c \odot d \rrbracket.
    \end{align}
    On the other hand, we see that
    \begin{align*}
        \{ \! \! \{ a, b, c \} \! \! \} \odot d = \{ \! \! \{ a, b, c \} \! \! \} \diamond d + d \diamond [a, b, c] + d \diamond [b, c, a] + d \diamond [c, a, b] + d \diamond \llbracket a, b, c \rrbracket + \{ \! \! \{ a, b, c \} \! \! \} * d
    \end{align*}
    and
    \begin{align*}
        c \odot \{ \! \! \{ a, b, d \} \! \! \} = c \diamond [a, b, d] + c \diamond [b, d, a] + c \diamond [d, a, b] + c \diamond \llbracket a, b, d \rrbracket + \{ \! \! \{ a, b, d \} \! \! \} \diamond c + c * \{ \! \! \{ a, b, d \} \! \! \}.
    \end{align*}
Hence, by expanding (\ref{leibniz-left-side}) using the identities (\ref{ns-np1}), (\ref{ns-np2}) and (\ref{ns-np3}), we get the Leibniz rule
\begin{align*}
     \{ \! \! \{ a, b, c \odot d \} \! \! \} = \{ \! \! \{ a, b, c \} \! \! \} \odot d  +  c \odot \{ \! \! \{ a, b, d \} \! \! \}.
\end{align*}
This concludes the proof.

    (ii) We have already seen that $(A; \mu_\diamond)$ is a representation of the commutative associative algebra $(A, \odot)$, and $(A; \rho_{[~,~,~]})$ is a representation of the $3$-Lie algebra $(A, \{ \! \! \{ ~, ~, ~ \} \! \! \})$. Thus, to show that $(A; \mu_\diamond, \rho_{[~,~,~]})$ is a representation of the subadjacent Nambu-Poisson algebra, we need to verify the identities in (\ref{np3-rep}). For any $a, b, c, d \in A$, we see that
    \begin{align*}
        \mu_\diamond ( \{ \! \! \{ a, b, c \} \! \! \}) d = \{ \! \! \{ a, b, c \} \! \! \} \diamond d 
        =~& [a, b, c \diamond d ] - c \diamond [a, b, d] \qquad (\text{by } (\ref{ns-np1})) \\
        =~& \big(  (\rho_{[~,~,~]} (a, b)) \mu_\diamond (c) - \mu_\diamond (c)   (\rho_{[~,~,~]} (a, b))  \big) d
    \end{align*}
    and
    \begin{align*}
        ( \rho_{[~,~,~]} (a, b \odot c)) d = [a, b \odot c, d] =~& b \diamond [a, c, d] + c \diamond [a, b, d] \qquad (\text{by }~ (\ref{ns-np2})) \\
        =~& \big(  \mu_\diamond (b) \rho_{[~,~,~]} (a, c) + \mu_\diamond (c) \rho_{[~,~,~]} (a, b)   \big)d.
    \end{align*}
    Hence, this part follows. Similarly, to show that $(\varphi_*, \psi_{\llbracket ~,~, ~ \rrbracket})$ is a Nambu-Poisson $2$-cocycle, we only need to verify the identity (\ref{2co-np}). For this, we observe that
    \begin{align*}
        &\psi_{\llbracket ~, ~, ~ \rrbracket} (a, b, c \odot d) + \rho_{[~,~,~]} (a, b) \varphi_* (c, d) \\
        &=\llbracket a, b, c \odot d \rrbracket + [a, b, c * d ] \\
        &= c* \{ \! \! \{ a, b, d \} \! \! \} + c \diamond \llbracket a, b, d \rrbracket + \{ \! \! \{ a, b, c \} \! \! \} * d + d \diamond \llbracket a, b, c \rrbracket \qquad (\text{by }~ (\ref{ns-np3})) \\
        &= \varphi_* (a, \{ \! \! \{ a, b, d \} \! \! \}) + \mu_\diamond (c) \psi_{\llbracket ~, ~, ~ \rrbracket} (a, b, d) + \varphi_* ( \{ \! \! \{ a, b, c \} \! \! \}, d) + \mu_\diamond (d) \psi_{\llbracket ~, ~, ~ \rrbracket} (a, b, c).
    \end{align*}
    Hence, this part also follows. This completes the proof.
\end{proof}

We will now show that NS-Nambu-Poisson algebras arise naturally from Nijenhuis operators on Nambu-Poisson algebras.

\begin{thm}\label{thm-nij-to-ns}
    Let $(A, ~ \! \cdot ~ \! , \{ ~,~,~ \})$ be a Nambu-Poisson algebra and $N: A \rightarrow A$ be a Nijenhuis operator on it. We define operations
    \begin{align*}
       & a \diamond_N b := N(a) \cdot b, \quad a *_N b := - N (a \cdot b), \quad [a, b, c]_N := \{ N(a), N(b), c \}\\
        &~~ \text{ and } ~~ \llbracket a, b, c \rrbracket_N := - N ( \{ N(a), b, c \} + \{ a, N(b), c \} + \{ a, b, N (c) \} - N \{ a, b, c \} ),
    \end{align*}
    for $a, b, c \in A.$ Then $(A, \diamond_N, *_N, [~,~,~]_N, \llbracket ~,~,~\rrbracket_N)$ is an NS-Nambu-Poisson algebra. Moreover, the corresponding subadjacent Nambu-Poisson algebra structure coincides with the deformed Nambu-Poisson algebra structure $A_N = (A, \cdot_N, \{ ~, ~, ~ \}_N)$ given in Theorem \ref{thm-deformed-np}.
\end{thm}

\begin{proof}
Since $N : A \rightarrow A$ is a Nijenhuis operator on the commutative associative algebra $(A, ~ \! \cdot ~ \!)$, it follows that $(A, \diamond_N, *_N)$ is an NS-commutative algebra \cite{BaishyaDas,leroux}. On the other hand, $N: A \rightarrow A$ is a Nijenhuis operator on the $3$-Lie algebra $(A, \{ ~,~,~\})$ implies that the triple $(A, [~,~,~]_N, \llbracket ~,~ ,~ \rrbracket_N)$ is an NS-$3$-Lie algebra \cite{ChtiouiHajjajiMabroukMakhlouf,Hou-Sheng-NSLie}. Thus, it remains to verify the identities (\ref{ns-np1}), (\ref{ns-np2}) and (\ref{ns-np3}). For any $a, b, c, d \in A$, we have
\begin{align*}
    \{ \! \! \{ a, b, c \} \! \! \} \diamond_N d = N ( \{ \! \! \{ a, b, c \} \! \! \}) \cdot d =~& \{ N (a), N (b), N (c) \} \cdot d \\
    =~& \{ N (a), N(b), N(c) \cdot d \} - N (c) \cdot \{ N (a) , N(b), d \} \\
    =~& [a, b, c \diamond_N d ]_N - c \diamond_N [a, b, d]_N,
\end{align*}
\begin{align*}
    [a \odot b, c, d ]_N = \{ N (a \odot b), N (c), d \} =~& \{ N (a) \cdot N (b), N (c), d  \} \\
   =~& N (a) \cdot \{ N(b), N(c), d \} + N (b) \cdot \{ N (a), N(c), d \} \\
   =~& a \diamond_N [b, c, d]_N + b \diamond_N [a, c, d]_N.
\end{align*}
Hence the identities (\ref{ns-np1}) and (\ref{ns-np2}) follow. Finally,
\begin{align*}
    &\llbracket a, b, c \odot d \rrbracket_N + [a, b, c *_N d ]_N - c *_N \{ \! \! \{ a, b, d \} \! \! \} - c \diamond_N \llbracket a, b, d \rrbracket_N - \{ \! \! \{ a, b, c \} \! \! \} *_N d - d \diamond_N \llbracket a, b, c \rrbracket_N \\
   & = -N \big(   \{ N (a), b , c \odot d \} + \{ a, N (b), c \odot d \} + \{ a, b, N (c \odot d) \} - N \{ a, b, c \odot d \} \big) \\
    & \quad - \{ N (a) ,N(b), N (c \cdot d) \} + N (c \cdot \{ \! \! \{ a, b, d \} \! \! \}) + N ( \{ \! \! \{ a, b, c \} \! \! \} \cdot d) \\
    & \quad + N (c) \cdot N ( \{ N (a), b, d \} + \{ a, N (b), d \} + \{ a, b, N (d) \} - N \{ a, b, d \} ) \\
    & \quad + N (d) \cdot N ( \{ N (a), b, c \} + \{ a, N (b), c \} + \{ a, b, N (c) \} - N \{ a,b, c \} ) \\
    &= N \Big( - \{ N (a), b, c \odot d \} - \{ a, N (b), c \odot d \} - \{ a, b, N (c) \cdot N(d) \} + N \{ a, b, c \odot d \} \\
    & \qquad \quad  - \{ N (a), N(b), c \cdot d \} - \{ N (a), b, N (c \cdot d) \} - \{ a, N (b), N (c \cdot d) \} \\
    & \qquad \quad  + N ( \{ N (a), b, c \cdot d \}+ \{ a, N (b), c \cdot d \} + \{ a, b, N (c \cdot d ) \} - N \{ a, b, c \cdot d \} ) + c \cdot \{ \! \! \{ a, b, d \} \! \! \} \\
    & \qquad \quad  + \{ \! \! \{ a, b, c \} \! \! \} \cdot d + N (c) \cdot ( \{ N (a), b, d \} + \{ a, N (b), d \} + \{ a, b, N (d) \}) - N(c) \cdot N (\{ a, b, d \}) \\
    &  \qquad \quad + c \cdot N ( \{ N (a), b, d \} + \{ a, N (b), d \} + \{ a, b, N (d) \} - N \{ a, b, d \}  ) \\
    & \qquad \quad - N ( c \cdot \{ N (a), b, d \} + c \cdot  \{ a, N (b), d \} + c \cdot  \{ a, b, N (d) \} - c \cdot  N \{ a, b, d \})\\
     &  \qquad \quad + N (d) \cdot (\{ N (a), b, c \} + \{ a, N (b), c \} + \{ a, b, N (c) \} ) - N(d) \cdot N (\{ a, b, c \}) \\
      &  \qquad \quad + d \cdot N ( \{ N (a), b, c \} + \{ a, N (b), c \} + \{ a, b, N (c) \} - N \{ a,b, c \} )\\
       &  \qquad \quad - N (d \cdot \{ N (a), b, c \} + d \cdot  \{ a, N (b), c \} + d \cdot  \{ a, b, N (c) \} - d \cdot  N \{ a,b, c \}) \Big).
\end{align*}
In the above expression, one may also expand the terms $N (c) \cdot N (\{ a, b, d \})$ and  $N (d) \cdot N (\{ a, b, c \})$ using the condition (\ref{nij-com}) of the Nijenhuis operator. Subsequently, by using the definition of the operations $\odot$ and $\{ \! \! \{ ~,~,~ \} \! \! \}$, and by using the Leibniz rule (\ref{leibniz-rule-added}) of the given Nambu-Poisson algebra, one observes that all the terms of the above expression are mutually cancelled with each other. Hence the identity (\ref{ns-np3}) also holds and thus $(A, \diamond_N, *_N, [~,~,~]_N, \llbracket ~, ~, ~ \rrbracket_N)$ is an NS-Nambu-Poisson algebra. Finally, for this structure, it is easy to see that $a \odot b = a \cdot_N b$ and $\{ \! \! \{ a, b, c \} \! \! \} = \{ a, b, c \}_N$, for $a, b, c \in A$. Thus, the subadjacent Nambu-Poisson algebra $(A, \odot, \{ \! \! \{ ~, ~, ~ \} \! \! \})$ coincides with the deformed Nambu-Poisson algebra $(A, \cdot_N, \{ ~,~, ~ \}_N)$.
\end{proof}

\begin{rmk}
    Let $(A, ~ \! \cdot ~ \!, \{ ~,~,~ \})$ be a Nambu-Poisson algebra and let $N : A \rightarrow A$ be any Nijenhuis operator on it. Then it follows from Proposition \ref{prop-higher-nij} and Theorem \ref{thm-nij-to-ns} that the tuple $(A, \diamond_{N^k}, *_{N^k}, [~,~,~]_{N^k} , \llbracket ~, ~, ~ \rrbracket_{N^k})$ is an NS-Nambu-Poisson algebra, for any $k \geq 0$.
\end{rmk}

Since a Nijenhuis operator on a Nambu-Poisson algebra induces an NS-Nambu-Poisson algebra structure, we obtain the following result by using Theorem \ref{ns-np-thm-2}.

\begin{pro}\label{nij-2co}
   Let $(A, ~ \! \cdot ~ \! , \{ ~, ~, ~ \})$ be a Nambu-Poisson algebra and $N : A \rightarrow A$ be a Nijenhuis operator on it. We define a linear map $\mu_N : A \rightarrow \mathrm{End}(A)$ and a skew-symmetric bilinear map $\rho_N : A \times A \rightarrow \mathrm{End}(A)$ by
   \begin{align*}
           \mu_N (a) b:= N (a) \cdot b ~~~~ \text{ and } ~~~~ \rho_N (a, b) c:= \{ N (a), N (b), c \}, \text{ for } a, b, c \in A.
       \end{align*}
       Then, the triple $(A; \mu_N, \rho_N)$ is a representation of the deformed Nambu-Poisson algebra $A_N = (A, ~ \! \cdot_N ~ \! , \{ ~, ~, ~ \}_N).$

       We also consider a bilinear map  $\varphi_N : A \times A \rightarrow A$ and a trilinear map $\psi_N : A \times A \times A \rightarrow A$ by
       \begin{align*}
           \varphi_N (a, b) = -N (a \cdot b), \quad \psi_N (a, b, c) = - N \big(  \{ N (a) , b, c \} + \{ a, N (b) , c \} + \{ a, b, N (c) \} - N \{ a, b, c \}  \big),
       \end{align*}
       for $a, b, c \in A$. Then the pair $(\varphi_N, \psi_N)$ is a Nambu-Poisson $2$-cocycle of the deformed Nambu-Poisson algebra $A_N$ with coefficients in the representation $(A; \mu_N, \rho_N)$. 
       %\item Finally, the identity map $\mathrm{Id} : A \rightarrow A_N$ is a $(\Phi_N, \Psi_N)$-twisted $\mathcal{O}$-operator on $A_N$ with respect to the representation $(A; \mu_N, \rho_N)$. Additionally, the induced Nambu-Poisson algebra structure on the representation space $A$ coincides with the deformed Nambu-Poisson structure $A_N$.
\end{pro}

We will return to this result when we define twisted $\mathcal{O}$-operators on a Nambu-Poisson algebra in the next section.

\section{Twisted $\mathcal{O}$-operators on Nambu-Poisson algebras}\label{section-5}

Twisted Rota-Baxter operators and more generally twisted $\mathcal{O}$-operators on associative algebras were first considered in \cite{uchino}. Later, these operators were generalized to Lie algebras \cite{das}, $3$-Lie algebras \cite{ChtiouiHajjajiMabroukMakhlouf,Hou-Sheng-NSLie} and Poisson algebras \cite{BaishyaDas}. In this section, we shall define these operators in the context of Nambu-Poisson algebras. In particular, we show that twisted $\mathcal{O}$-operators on a Nambu-Poisson algebra are closely related to NS-Nambu-Poisson algebras considered in the previous section.

\begin{defi}
Let $(A, ~ \! \cdot ~ \!, \{ ~, ~, ~ \})$ be a Nambu-Poisson algebra and $(V; \mu, \rho)$ be a representation of it. Suppose $(\varphi, \psi)$ is a Nambu-Poisson $2$-cocycle of $A$ with coefficients in the representation $V$. Then a linear map $r : V \rightarrow A$ is said to be a {\bf $(\varphi, \psi)$-twisted $\mathcal{O}$-operator} on $A$ with respect to the representation $V$ if it satisfies
\begin{align}
    r(u) \cdot r(v) =~& r \big( \mu (r(u)) v + \mu (r(v)) u + \varphi (r(u), r (v) ) \big), \label{first-twisted}\\
    \{ r(u), r(v) , r(w) \} =~& r \big(  \rho (r(u), r(v)) w + \rho (r(v), r(w)) u + \rho ( r(w), r (u)) v + \psi (r(u), r(v), r(w)) \big), \label{second-twisted}
\end{align}
for all $u, v, w \in V$.
\end{defi}

It follows from (\ref{first-twisted}) that $r$ is a $\varphi$-twisted $\mathcal{O}$-operator on the commutative associative algebra $(A, ~ \! \cdot ~ \!)$ with respect to the representation $(V; \mu)$. While, the condition (\ref{second-twisted}) means that $r$ is also a $\psi$-twisted $\mathcal{O}$-operator on the $3$-Lie algebra $(A, \{ ~, ~, ~ \})$ with respect to the representation $(V; \rho)$.

\begin{rmk}\label{remark-tw-rota}
\begin{itemize}
    \item[(i)]  When $(V; \mu, \rho) = (A; \mu_\mathrm{ad}, \rho_{\mathrm{ad}})$ is the adjoint representation of the Nambu-Poisson algebra $(A, ~\! \cdot ~ \!, \{ ~, ~, ~\})$, and $(\varphi, \psi)$ is a Nambu-Poisson $2$-cocycle of $A$ with coefficients in the adjoint representation, a $(\varphi, \psi)$-twisted $\mathcal{O}$-operator is simply called a {\bf $(\varphi, \psi)$-twisted Rota-Baxter operator} on the Nambu-Poisson algebra $(A, ~ \! \cdot ~ \!, \{ ~, ~, ~ \})$.
    
    \item[(ii)] The notion of twisted $\mathcal{O}$-operators on a Poisson algebra with respect to a representation was considered in \cite{BaishyaDas} while studying NS-Poisson algebras. We have already seen that fixing a coordinate of the underlying $3$-Lie bracket of a Nambu-Poisson algebra, one obtains a Poisson algebra. A similar result is also obtained for the corresponding representations and $2$-cocycles (cf. Proposition \ref{fixing-co} and Proposition \ref{2co-prop}). Here we shall show that a twisted $\mathcal{O}$-operator on a Nambu-Poisson algebra is also a twisted $\mathcal{O}$-operator on the Poisson algebra (obtained by fixing a coordinate) under some mild condition. More precisely, let $r: V \rightarrow A$ be a $(\varphi, \psi)$-twisted $\mathcal{O}$-operator on the Nambu-Poisson algebra $(A, ~ \! \cdot ~ \!, \{ ~, ~, ~ \})$ with respect to the representation $(V; \mu, \rho)$. Suppose \begin{align*}
        u_0 \in V ~\text{ be such that }~ \rho (r(u), r(v)) u_0 = 0, \text{ for all } u, v \in V
    \end{align*}
    and take $x_0 = r(u_0) \in A$. Then $r: V \rightarrow A$ is also a $(\varphi, \psi_{x_0})$-twisted $\mathcal{O}$-operator on the Poisson algebra $(A, ~ \! \cdot ~ \! , \{ ~,~ \}_{x_0})$ with respect to the representation $(V; \mu, \rho_{x_0})$. To see this, we observe that
     \begin{align*}
        &\{ r(u), r(v) \}_{x_0} \\
        &= \{ x_0, r(u), r(v) \} \\
        &= r \big(   \rho ( r(u_0), r(u)) v+ \rho (r(u), r(v)) u_0 + \rho (r(v), r (u_0)) u + \psi (r(u_0), r(u), r(v))   \big) \\
        &= r \big(  \rho_{x_0} (r(u)) v - \rho_{x_0} (r(v)) u + \psi_{x_0} (r(u), r(v)) \big),
    \end{align*}
    for any $u, v \in V$. Hence, the claim follows.
\end{itemize}
\end{rmk}

In the following result, we characterize twisted $\mathcal{O}$-operators in terms of their graphs.

\begin{pro}
    Let $(A, ~\! \cdot ~ \!, \{ ~, ~, ~\})$ be a Nambu-Poisson algebra, $(V; \mu, \rho)$ be a representation and $(\varphi, \psi)$ be a Nambu-Poisson $2$-cocycle of $A$ with coefficients in $V$. Then a linear map $r : V \rightarrow A$ is a $(\varphi, \psi)$-twisted $\mathcal{O}$-operator on $A$ with respect to the representation $V$ if and only if the graph
    \begin{align*}
        Graph (r) = \{   (r(u), u) ~ \! | ~ \! u \in V \}
    \end{align*}
    is a subalgebra of the $(\varphi, \psi)$-twisted semidirect product Nambu-Poisson algebra $A \ltimes_{ (\varphi, \psi)} V.$
\end{pro}

\begin{proof}
Let $(r(u), u), (r(v), v) , (r(w), w) \in Graph (r)$ be arbitrary. Then from (\ref{eqn-semid-2co}) and (\ref{eqn-semid-2co-lie}), we have
\begin{align*}
    (r(u), u) \cdot_{\ltimes_\varphi} (r(v), v) = \big(  r(u) \cdot r(v) ~\! , ~ \! \mu({r(u)}) v + \mu (r(v)) u + \varphi (r(u), r (v)) \big)
\end{align*}
and 
\begin{align*}
    &\{  (r(u), u), (r(v), v) , (r(w), w) \}_{\ltimes_\psi} \\
    & \quad = \big( \{ r(u), r(v), r(w) \} ~\!, ~ \! \rho ( r(u), r(v)) w + \rho (r(v), r (w)) u + \rho (r (w), r(u)) v + \psi (r(u), r(v), r(w))  \big).
\end{align*}
Both of these elements are again in $Graph (r)$ if and only if the identities (\ref{first-twisted}) and (\ref{second-twisted}) hold. Hence, the result follows.
\end{proof}

As a consequence of the above proposition, we have the following result.

\begin{pro}\label{prop-twist-ns}
    Let $r: V \rightarrow A$ be a $(\varphi, \psi)$-twisted $\mathcal{O}$-operator on the Nambu-Poisson algebra $(A, ~ \! \cdot ~ \! , \{ ~, ~, ~ \})$ with respect to the representation $(V; \mu, \rho)$. Then $V$ inherits a Nambu-Poisson algebra structure whose multiplication and the $3$-Lie bracket are respectively given by
    \begin{align*}
        u \cdot_r v :=~& \mu({r(u)}) v + \mu (r(v)) u + \varphi (r(u), r (v)),\\
        \{ u, v, w \}_r :=~& \rho ( r(u), r(v)) w + \rho (r(v), r (w)) u + \rho (r (w), r(u)) v + \psi (r(u), r(v), r(w)),
    \end{align*}
    for $u, v, w \in V$. This structure is said to be induced by $r$. With this structure, $r: V \rightarrow A$ is a homomorphism of Nambu-Poisson algebras.
\end{pro}

In the next result, we show that the Nambu-Poisson algebra obtained in the above proposition can be seen as the subadjacent Nambu-Poisson algebra of an NS-Nambu-Poisson algebra. More precisely, we have the following.

\begin{thm}\label{twisted-NS}
    Let $r: V \rightarrow A$ be a $(\varphi, \psi)$-twisted $\mathcal{O}$-operator on the Nambu-Poisson algebra $(A, ~ \! \cdot ~ \! , \{ ~, ~, ~ \})$ with respect to the representation $(V; \mu, \rho)$. Then $V$ carries an NS-Nambu-Poisson algebra structure $(V, \diamond_r, *_r, [~,~,~]_r, \llbracket ~, ~, ~ \rrbracket_r)$, where
    \begin{align*}
        &u \diamond_r v :=  \mu (r(u)) v, \quad u *_r v := \varphi (r(u), r(v)), \quad [u, v, w ]_r : = \rho (r(u), r(v)) w  \\
       & \qquad \qquad \text{ and } ~ \llbracket u, v, w \rrbracket_r := \psi (r(u), r(v), r(w)), \text{ for } u, v, w \in V.
    \end{align*}
    Moreover, the corresponding subadjacent Nambu-Poisson algebra coincides with the Nambu-Poisson algebra $(V, \cdot_r, \{ ~, ~, ~ \}_r)$ obtained in the previous proposition.
\end{thm}

\begin{proof}
    Since $r: V \rightarrow A$ is a $\varphi$-twisted $\mathcal{O}$-operator on the commutative associative algebra $(A, ~ \! \cdot ~ \!)$ with respect to the representation $(V; \mu)$, it follows that $(V, \diamond_r, *_r)$ is an NS-commutative algebra \cite{BaishyaDas}. On the other hand, $r: V \rightarrow A$ is also a $\psi$-twisted $\mathcal{O}$-operator on the $3$-Lie algebra $(A, \{ ~,~, ~ \})$ with respect to the representation $(V; \rho)$ implies that $(V, [~,~,~]_r, \llbracket ~, ~, ~ \rrbracket_r)$ is an NS-$3$-Lie algebra \cite{ChtiouiHajjajiMabroukMakhlouf,Hou-Sheng-NSLie}. Next, for any $u, v, w, w' \in V$, we observe that
    \begin{align*}
        \{ \! \! \{ u, v, w \} \! \! \} \diamond_r w' = \mu \big( r (  \{ \! \! \{ u, v, w \} \! \! \}) \big) w' =~& \mu ( \{ r(u), r(v), r(w) \}) w' \\
        =~& \rho (r(u), r(v)) \mu (r(w)) w' - \mu (r(w)) \rho (r(u), r(v)) w' \\
        =~& [u, v, w \diamond_r w']_r - w \diamond_r [u, v, w']_r,
    \end{align*}
    \begin{align*}
        [u \odot v , w, w']_r = \{ r (u \odot v), r(w), w' \} =~& \{ r(u) \cdot r(v), r(w) , w' \} \\
        =~& r(u) \cdot \{ r(v), r(w), w' \} + r(v) \cdot \{ r(u), r(w), w' \} \\
        =~& u \diamond_r [v, w, w']_r + v \diamond_r [u, w, w']_r.
    \end{align*}
    and finally,
    \begin{align*}
        &\llbracket u, v, w \odot w' \rrbracket_r + [u, v, w *_r w' ]_r \\
        &= \psi (r(u), r(v), r (w \odot w')) + \rho ( r(u), r(v)) (w *_r w') \\
        &= \psi (r(u), r(v), r(w) \cdot r(w')) + \rho ( r(u), r(v)) \varphi (r(w), r(w')) \\
        &= \varphi ( r(w), \{ r(u), r(v), r(w') \}) + \mu(r(w)) \psi ( r(u), r(v), r(w') ) \\
        & \qquad + \varphi ( \{ r(u), r(v) , r(w) \}, r(w') ) + \mu (r(w')) \psi ( r(u), r(v), r(w)) \\
        &= w *_r \{ \! \! \{ u, v, w' \} \! \! \} + w \diamond_r \llbracket u, v, w' \rrbracket_r + \{ \! \! \{ u, v, w \} \! \! \} *_r w' + w' \diamond_r \llbracket u, v, w \rrbracket.
    \end{align*}
    This verifies the identities (\ref{ns-np1}), (\ref{ns-np2}) and (\ref{ns-np3}). Hence $(V, \diamond_r, *_r , [~,~,~]_r, \llbracket ~, ~, ~ \rrbracket_r)$ is an NS-Nambu-Poisson algebra. Finally, for this structure, it is easy to see that
    \begin{align*}
        u \odot v =~& u \diamond_r v + v \diamond_r u + u *_r v = u \cdot_r v,\\
        \{ \! \! \{ u, v, w \} \! \! \{ =~& [u, v, w]_r + [v, w, u]_r + [w, u, v]_r + \llbracket u, v, w \rrbracket_r = \{ u, v, w \}_r.
    \end{align*}
    Hence the corresponding subadjacent structure coincides with $(V, \cdot_r, \{ ~,~,~ \}_r)$.
\end{proof}

In the previous theorem, we have seen that a twisted $\mathcal{O}$-operator induces an NS-Nambu-Poisson algebra structure. We will now show that any NS-Nambu-Poisson algebra always arises in this way. Explicitly, we have the following result.

\begin{pro}
    Let $(A, \diamond, *, [~,~, ~], \llbracket ~, ~, ~ \rrbracket)$ be any NS-Nambu-Poisson algebra. Then the identity map $\mathrm{Id} : A \rightarrow A$ is a $(\varphi_*, \psi_{\llbracket ~,~, ~ \rrbracket})$-twisted $\mathcal{O}$-operator on the subadjacent Nambu-Poisson algebra $(A, \odot, \{ \! \! \{ ~, ~, ~ \} \! \! \})$ with respect to the representation $(A; \mu_\diamond, \rho_{[~,~,~]})$. Moreover, the induced NS-Nambu-Poisson algebra structure on $A$ coincides with the given one.
\end{pro}

\begin{proof}
    For any $a, b, c \in A$,
    \begin{align*}
        \mathrm{Id} (a) \odot \mathrm{Id} (b)  =~& a \odot b = a \diamond b + b  \diamond a + a * b\\
        =~& \mathrm{Id} \big( \mu_\diamond (\mathrm{Id}(a)) b + \mu_\diamond (\mathrm{Id}(b)) a + \varphi_* (\mathrm{Id}(a), \mathrm{Id}(b))  \big),
    \end{align*}
    and similarly, 
    \begin{align*}
        &\{ \! \! \{ \mathrm{Id} (a), \mathrm{Id}(b), \mathrm{Id}(c) \} \! \! \} = \{ \! \! \{ a, b, c \} \! \! \} \\
        &= \mathrm{Id} \big( \rho_{[~,~,~]} ( \mathrm{Id} (a), \mathrm{Id} (b)   )c +  \rho_{[~,~,~]} ( \mathrm{Id} (b), \mathrm{Id} (c)   )a + \rho_{[~,~,~]} ( \mathrm{Id} (c), \mathrm{Id} (a)   )b + \psi_{\llbracket ~, ~, ~ \rrbracket} (    \mathrm{Id} (a), \mathrm{Id}(b), \mathrm{Id} (c) )   \big).
    \end{align*}
    This shows that the identity map $\mathrm{Id}: A \rightarrow A$ is a $(\varphi_*, \psi_{\llbracket ~, ~, ~ \rrbracket})$-twisted $\mathcal{O}$-operator on the Nambu-Poisson algebra $(A, \odot, \{ \! \! \{ ~, ~, ~ \} \! \! \})$ with respect to the representation $(A; \mu_\diamond, \rho_{[~,~,~]})$. If $(A, \diamond_\mathrm{Id}, *_\mathrm{Id}, [~,~,~]_\mathrm{Id}, \llbracket ~, ~, ~ \rrbracket_\mathrm{Id})$ is the induced NS-Nambu-Poisson structure on $A$ (by Theorem \ref{twisted-NS}) then
    \begin{align*}
        &a \diamond_\mathrm{Id} b = \mu_\diamond (\mathrm{Id}(a)) b = a \diamond b, \quad a *_\mathrm{Id} b = \varphi_* (\mathrm{Id}(a), \mathrm{Id} (b) ) = a * b \\
        \text{ and} & \text{ similarly,~}~ [a, b, c]_\mathrm{Id} = [a, b, c] \text{ and } \llbracket a, b, c \rrbracket_\mathrm{Id} = \llbracket a, b, c \rrbracket, \text{ for } a, b, c \in A.
    \end{align*}
    This proves the last part.
\end{proof}

By applying the above result to the NS-Nambu-Poisson algebra obtained from a Nijenhuis operator (cf. Theorem \ref{thm-nij-to-ns} and Proposition \ref{nij-2co}), we get the following.

\begin{pro}
    Let $(A, ~ \! \cdot ~ \! , \{ ~, ~, ~\})$ be a Nambu-Poisson algebra and $N : A \rightarrow A$ be a Nijenhuis operator on it. Consider the deformed Nambu-Poisson algebra $A_N = (A, \cdot_N , \{ ~, ~, ~ \}_N)$, its representation $(A; \mu_N, \rho_N)$ and the Nambu-Poisson $2$-cocycle $(\varphi_N, \psi_N)$ described in Proposition \ref{nij-2co}. With this, the identity map $\mathrm{Id}: A \rightarrow A$ is a $(\varphi_N, \psi_N)$-twisted $\mathcal{O}$-operator. Moreover, the induced NS-Nambu-Poisson algebra structure on $A$ coincides with the one given in Theorem \ref{thm-nij-to-ns}.
\end{pro}

In Theorem \ref{ns-np-thm-2}, we have seen that an NS-Nambu-Poisson algebra gives rise to its subadjacent Nambu-Poisson algebra. The next result finds a necessary and sufficient condition under which an arbitrary Nambu-Poisson algebra arises as the subadjacent structure of an NS-Nambu-Poisson algebra.

\begin{pro}
    Let $(A, ~ \! \cdot ~ \!, \{ ~, ~, ~ \})$ be a Nambu-Poisson algebra. Then it is the subadjacent Nambu-Poisson algebra of an NS-Nambu-Poisson structure on $A$ if and only if there exists an invertible $(\varphi,\psi)$-twisted $\mathcal{O}$-operator $r: V \rightarrow A$ on the Nambu-Poisson algebra $(A, ~ \! \cdot ~ \!, \{ ~, ~, ~ \})$ with respect to a representation $(V; \mu, \rho)$ and a Nambu-Poisson $2$-cocycle $(\varphi, \psi)$. In this case, the NS-Nambu-Poisson algebra structure on $A$ is given by 
    \begin{align*}
        &a \diamond b = r ( \mu (a) r^{-1} (b)), \quad a * b = r (\varphi (a, b)), \quad [a, b, c ] = r (\rho (a, b) r^{-1} (c))\\
        &  \qquad \qquad \text{and } ~ \llbracket a, b, c \rrbracket = r (\psi (a, b, c)), \text{ for } a, b, c \in A.
    \end{align*}
\end{pro}

\begin{proof}
    Let $(A, \diamond, *, [~,~,~], \llbracket ~, ~, ~ \rrbracket)$ be an NS-Nambu-Poisson structure on $A$ whose subadjacent Nambu-Poisson algebra structure is given by $(A, ~ \! \cdot ~ \! , \{ ~, ~, ~ \})$. Then by Theorem \ref{ns-np-thm-2}, the triple $(A; \mu_\diamond, \rho_{[~,~,~]})$ is a representation of the Nambu-Poisson algebra $(A, ~ \! \cdot ~ \!, \{ ~, ~, ~ \})$, and $(\varphi_*, \psi_{\llbracket ~, ~, ~ \rrbracket})$ is a Nambu-Poisson $2$-cocycle. Additionally, we see that
    \begin{align*}
        \mathrm{Id} (a) \cdot \mathrm{Id} (b) = a \cdot b = a \diamond b + b \diamond a + a * b =  \mathrm{Id} \big( \mu_\diamond (\mathrm{Id} (a)) b  + \mu_\diamond (\mathrm{Id} (b)) a + \varphi_* (\mathrm{Id}(a), \mathrm{Id}(b) ) \big) 
    \end{align*}
    and 
    \begin{align*}
       & \{  \mathrm{Id} (a),  \mathrm{Id} (b),  \mathrm{Id} (c) \} = \{ a, b, c \} = [a, b, c] + [b, c, a] + [c, a, b] + \llbracket a, b, c \rrbracket \\
       &=  \mathrm{Id}  \Big(   \rho_{[~,~,~]} (  \mathrm{Id} (a),  \mathrm{Id} (b) ) c +  \rho_{[~,~,~]} ( \mathrm{Id} (b),  \mathrm{Id} (c) ) a +  \rho_{[~,~,~]} (\mathrm{Id} (c),  \mathrm{Id} (a) ) b + \psi_{\llbracket ~,~,~ \rrbracket} ( \mathrm{Id}(a), \mathrm{Id}(b), \mathrm{Id}(c)) \Big),
    \end{align*}
    for all $a, b, c \in A$. Hence $\mathrm{Id}: A \rightarrow A$ is an invertible $(\varphi_*, \psi_{\llbracket ~, ~, ~ \rrbracket})$-twisted $\mathcal{O}$-operator on $(A, ~ \! \cdot ~ \!, \{ ~, ~, ~ \})$ with respect to the representation $(A; \mu_\diamond, \rho_{[~,~,~]}).$

    Conversely, let $r: V \rightarrow A$ be an invertible $(\varphi, \psi)$-twisted $\mathcal{O}$-operator on $(A, ~ \! \cdot ~ \!, \{ ~, ~, ~ \})$ with respect to a representation $(V; \mu, \rho)$ and a Nambu-Poisson $2$-cocycle $(\varphi, \psi)$. Then by Theorem \ref{twisted-NS}, $V$ inherits an NS-Nambu-Poisson algebra structure $(V, \diamond_r, *_r, [~,~,~]_r, \llbracket ~, ~, ~ \rrbracket_r)$, where the operations are explicitly given there. Since $r$ is invertible, we can transfer this NS-Nambu-Poisson algebra structure to the vector space $A$. Explicitly, for any $a, b, c \in A$, we have
    \begin{align*}
        &a \diamond b = r ( r^{-1} (a) \diamond_r r^{-1}(b)) = r (\mu(a) r^{-1}(b)), \quad a * b = r ( r^{-1}(a) *_r r^{-1}(b) ) = r (\varphi (a, b)),\\
        & \quad [a, b, c ] = r ( [r^{-1}(a), r^{-1}(b), r^{-1}(c)]_r) = r (\rho (a, b) r^{-1}(c)),\\
        & \quad \llbracket a, b, c \rrbracket = r ( \llbracket r^{-1}(a), r^{-1}(b), r^{-1}(c) \rrbracket_r ) = r (\psi (a, b, c)).
    \end{align*}
    If $(A, \odot, \{ \! \! \{ ~, ~, ~\} \! \! \})$ is the corresponding subadjacent Nambu-Poisson algebra structure then 
    \begin{align*}
        a \odot b = a \diamond b + b \diamond a + a * b =~& r \big(  \mu(a) r^{-1} (b) + \mu (b) r^{-1} (a) + \varphi (a, b)    \big) \\
        =~&  r (r^{-1}(a)) \cdot r (r^{-1}(b)) = a \cdot b
    \end{align*}
    and similarly, $\{ \! \! \{ a, b, c \} \! \! \} = \{ a, b, c \}$, for all $a, b, c \in A$. Hence $(A, \diamond, *, [~,~,~], \llbracket ~, ~, ~\rrbracket)$ is a desired NS-Nambu-Poisson algebra structure on $A$.
\end{proof}

The notion of a Reynolds operator is a particular case of a twisted Rota-Baxter operator. Hence, Reynolds operators are closely related to NS-algebras \cite{uchino,das,Hou-Sheng-NSLie}. We will now consider the Reynolds operator on a Nambu-Poisson algebra and find its relation to NS-Nambu-Poisson algebras.

\begin{defi}
    Let $(A, ~ \! \cdot ~ \!, \{ ~, ~, ~\})$ be a Nambu-Poisson algebra. A linear map $\mathcal{R} : A \rightarrow A$ is said to be a {\bf Reynolds operator} on $A$ if for all $a, b, c \in A$,
    \begin{align*}
        \mathcal{R}(a) \cdot \mathcal{R}(b) =~& \mathcal{R} \big( \mathcal{R}(a) \cdot b + a \cdot \mathcal{R}(b) - \mathcal{R}(a) \cdot \mathcal{R}(b) \big),\\
        \{ \mathcal{R}(a), \mathcal{R}(b), \mathcal{R}(c) \} =~& \mathcal{R} \big(  \{ \mathcal{R}(a), \mathcal{R}(b), c \} + \{ \mathcal{R}(a), b, \mathcal{R}(c) \} + \{ a, \mathcal{R}(b), \mathcal{R}(c) \} - \{ \mathcal{R}(a), \mathcal{R}(b), \mathcal{R}(c) \}   \big).
    \end{align*}
\end{defi}

It follows that a linear map $\mathcal{R}: A \rightarrow A$ is a Reynolds operator on the Nambu-Poisson algebra $(A, ~ \! \cdot ~ \!, \{ ~, ~, ~\})$ if and only if $\mathcal{R}$ is a Reynolds operator on both the commutative associative algebra $(A, ~\! \cdot ~ \!)$ and on the $3$-Lie algebra $(A, \{ ~, ~, ~\})$.

\begin{rmk}
    Let $(A, ~ \! \cdot ~ \!, \{ ~, ~, ~\})$ be a Nambu-Poisson algebra and consider its adjoint representation $(A; \mu_\mathrm{ad}, \rho_\mathrm{ad})$. Then it is easy to see that the pair $(\varphi_{-}, \psi_{-})$ is a Nambu-Poisson $2$-cocycle of $A$ with coefficients in the adjoint representation, where
    \begin{align*}
        \varphi_{- }(a, b) = - a \cdot b ~~~~ \text{ and } ~~~~ \psi_{-} (a, b, c) = - \{ a, b, c \}, \text{ for } a, b, c \in A. 
    \end{align*}
    It follows that a Reynolds operator on a Nambu-Poisson algebra $(A, ~\! \cdot ~ \! , \{ ~, ~, ~\})$ is nothing but a $(\varphi_{-}, \psi_{-})$-twisted Rota-Baxter operator on A (cf. Remark \ref{remark-tw-rota}).
\end{rmk}

Since a Reynolds operator can be realized as a twisted Rota-Baxter operator, we have the following result (as a consequence of Proposition \ref{prop-twist-ns} and Theorem \ref{twisted-NS}).

\begin{pro}
    Let $(A, ~ \! \cdot ~ \!, \{ ~, ~, ~\})$ be a Nambu-Poisson algebra and $\mathcal{R}: A \rightarrow A$ be a Reynolds operator on it.
    \begin{itemize}
        \item[(i)] Then $(A, \cdot_\mathcal{R}, \{ ~, ~, ~ \}_\mathcal{R})$ is a Nambu-Poisson algebra, where for any $a, b, c \in A$,
        \begin{align*}
            a \cdot_\mathcal{R} b :=~& \mathcal{R}(a) \cdot b + a \cdot \mathcal{R}(b) - \mathcal{R}(a) \cdot \mathcal{R}(b),\\
            \{ a, b, c \}_\mathcal{R} :=~& \{ \mathcal{R}(a), \mathcal{R}(b), c \} + \{ \mathcal{R}(a), b, \mathcal{R}(c) \} + \{ a, \mathcal{R}(b), \mathcal{R}(c) \} - \{ \mathcal{R}(a), \mathcal{R}(b), \mathcal{R}(c) \}.
        \end{align*}
        \item[(ii)] The quintuple $(A, \diamond_\mathcal{R}, *_\mathcal{R}, [~,~,~]_\mathcal{R}, \llbracket ~, ~, ~\rrbracket_\mathcal{R})$ is an NS-Nambu-Poisson algebra, where
        \begin{align*}
            &a \diamond_\mathcal{R} b := \mathcal{R}(a) \cdot b, \quad a *_\mathcal{R}b:= - \mathcal{R}(a) \cdot \mathcal{R} (b), \quad [a, b, c ]_\mathcal{R} := \{ \mathcal{R}(a) , \mathcal{R}(b), c \} \\
            & \qquad \text{ and } \llbracket a, b, c \rrbracket_\mathcal{R} := - \{ \mathcal{R}(a) , \mathcal{R}(b), \mathcal{R}(c) \},
        \end{align*}
        for $a, b, c \in A$. The corresponding subadjacent Nambu-Poisson algebra is precisely the one $(A, \cdot_\mathcal{R}, \{ ~, ~, ~ \}_\mathcal{R})$ given above.
    \end{itemize}
\end{pro}

\noindent {\bf Further discussions.}
In \cite{Flato}, Flato, Gerstenhaber and Voronov developed the cohomology theory of a Poisson algebra using an innovative bicomplex that combines the Harrison cochain complex of the underlying commutative associative algebra and the Chevalley-Eilenberg complex of the Lie algebra. It has been shown that the corresponding second cohomology groups classify infinitesimal deformations of Poisson algebras. Later, Bao and Ye  \cite{Bao1, Bao2} observed that the Flato-Gerstenhaber-Voronov cohomology of a Poisson algebra can be described by using classical Yoneda extensions or by derived functors. It is worth mentioning that Lichnerowicz \cite{Lichnerowicz} considered a different cohomology theory for Poisson algebras before the paper by Flato, Gerstenhaber and Voronov \cite{Flato} appeared. Unlike the Flato-Gerstenhaber-Voronov cohomology (which governs the simultaneous deformations of both the commutative associative product and the Lie bracket), the Lichnerowicz cohomology only captures the deformations of the Lie bracket. In \cite{Hanene}, the author proposed a cohomology theory for Nambu-Poisson algebras generalizing the Lichnerowicz cohomology of Poisson algebras. Thus, their cohomology only captures the deformations of the underlying $3$-Lie bracket. However, the more general cohomology theory of Nambu-Poisson algebras (analogous to the Flato-Gerstenhaber-Voronov cohomology of Poisson algebras) is not yet considered. In a subsequent work, we aim to construct a suitable bicomplex combining the Harrison cochain complex of the underlying commutative associative algebra and the cochain complex of the $3$-Lie algebra (as described in \cite{Takhtajan}). The resulting cohomology should also generalizes the second cohomology group of a Nambu-Poisson algebra considered in the present paper.

Like twisted Rota-Baxter operators are related to NS-algebras, Rota-Baxter operators of weight $1$ are related to post-algebras \cite{li-guo}. They are also intrinsically associated with corresponding bialgebra structures. In \cite{poisson-bi}, the authors considered Rota-Baxter operators of weight $1$ on Poisson algebras, which led them to introduce post-Poisson algebras and Poisson bialgebras. In a forthcoming paper, we aim to extend their study to the context of Nambu-Poisson algebras. Explicitly, we develop the notions of post-Nambu-Poisson algebras, matched pairs and Manin triples of Nambu-Poisson algebras, and their relations with Nambu-Poisson bialgebras. We also aim to investigate the connections between the solutions of the Nambu-Poisson Yang-Baxter equation considered in \cite{HarrathiSendi} and Nambu-Poisson bialgebras.

\end{document}